\documentclass[11pt]{amsart}
\usepackage{indentfirst,latexsym,bm}
\usepackage{amsfonts}
\usepackage{amssymb}
\usepackage{amsmath}
\usepackage{hyperref}
\usepackage{dsfont}
\usepackage{leftidx}
\usepackage{amsthm}
\allowdisplaybreaks[4]
\usepackage{geometry}
\usepackage{amsmath,amscd}
\usepackage{shuffle}
\usepackage[all,cmtip]{xy}
\usepackage{color,xcolor}
\usepackage[all]{xy}

\newtheorem{theorem}{Theorem}[section]
\newtheorem{proposition}[theorem]{Proposition}
\newtheorem{lemma}[theorem]{Lemma}
\newtheorem{corollary}[theorem]{Corollary}
\theoremstyle{definition}
\newtheorem{definition}[theorem]{Definition}
\newtheorem{example}[theorem]{Example}
\newtheorem{remark}[theorem]{Remark}

\newcommand{\B}{\mathcal{B}}

\newcommand{\T}{\mathcal{T}}
\newcommand{\Tt}{\overline{\mathcal{T}}}
\newcommand{\K}{\mathds{k}}
\newcommand{\BN}{\mathcal{B}}

\newcommand{\Pp}{\mathcal{P}}
\newcommand{\cL}{\mathcal{L}}
\newcommand{\cI}{\mathcal{I}}
\newcommand{\cF}{\mathcal{F}}
\newcommand{\bN}{\mathbb{N}}

\newcommand\id{\operatorname{id}}
\newcommand\Char{\operatorname{char}}
\newcommand\gr{\operatorname{gr}}

\newcommand\GK{\operatorname{GK}}

\newcommand\corad{\operatorname{corad}}
\newcommand{\Z}{\mathbb{Z}}
\newcommand{\G}{\mathbf{G}}
\newcommand{\st}{\operatorname{st}}
\newcommand{\X}{\langle X \rangle}

\theoremstyle{plain}
\newcounter{maint}

\def\BFdB{\mathcal{B}_{\operatorname{FdB}}}
\def\BFdBnc{\mathcal{B}_{\operatorname{FdB}}^{\operatorname{nc}}}
\def\HFdB{\mathcal{H}_{\operatorname{FdB}}}
\begin{document}
\thispagestyle{empty}
\title[Quotient Hopf algebras of the free bialgebra with PBW bases and GK-dimensions]
{Quotient Hopf algebras of the free bialgebra with PBW bases and GK-dimensions}

\author{
{Huan Jia $^{1,2}$},
{Naihong Hu $^{1}$}, 
{Rongchuan Xiong $^3$} and 
{Yinhuo Zhang $^{2,*}$}}

\address{$^{1}$ School of Mathematical Sciences, East China Normal University, Shanghai 200241, China}

\address{$^{2}$ Department of Mathematics and Statistics,
University of Hasselt, Universitaire Campus, 3590 Diepenbeek, Belgium}

\address{$^{3}$ Department of Mathematics, Changzhou University, Changzhou 213164, China}

\address{Huan Jia $^{1,2}$}
\email{huan.jia@uhasselt.be}

\address{Naihong Hu $^{1}$}
\email{nhhu@math.ecnu.edu.cn}

\address{Rongchuan Xiong $^3$}
\email{rcxiong@foxmail.com}

\address{Yinhuo Zhang $^{2,*}$}
\email{yinhuo.zhang@uhasselt.be}


\date{}

\maketitle

\begin{abstract}
Let $\K$ be a field. We study the free bialgebra $\T$  generated by the coalgebra $C=\K g\oplus \K h$ and its
quotient bialgebras (or Hopf algebras)  over $\K$. We show that the free noncommutative Fa\`a di Bruno bialgebra is a sub-bialgebra of $\T$, and the quotient bialgebra $\Tt:=\T/(E_{\alpha}|~\alpha(g)\ge 2)$ is an Ore extension of the well-known Fa\`a di Bruno bialgebra. The image of the free noncommutative Fa\`a di Bruno bialgebra in the quotient $\Tt$ gives a more reasonable non-commutative version of the commutative Fa\`a di Bruno bialgebra from the PBW basis point view. 
If $\Char\K=p>0$, we obtain a chain of quotient Hopf algebras of $\Tt$: $\Tt\twoheadrightarrow  \Tt_{n}\twoheadrightarrow \Tt_{n}'(p)\twoheadrightarrow \Tt_{n}(p)\twoheadrightarrow \Tt_{n}(p;d_{1}) 
\twoheadrightarrow\ldots \twoheadrightarrow \Tt_{n}(p;d_{j},d_{j-1},\ldots,d_{1})\twoheadrightarrow \ldots \twoheadrightarrow \Tt_{n}(p;d_{p-2},d_{p-3},\ldots,d_{1})$ with finite GK-dimensions.  Furthermore, we study the homological properties and the coradical filtrations of those quotient Hopf algebras. 
\end{abstract}

\section*{Introduction}

Hopf algebras arise from many different aspects, e.g. groups, dg manifolds, Lie algebras, lattices, graphs, braided spaces etc. One of pure structure constructions of Hopf algebras is to generate a free Hopf algebra (or a bialgebra) by a coalgebra \cite{T1971}. Such a free Hopf algebra  is usually too big to have nice properties. So we consider its sub-Hopf algebras and quotient Hopf algebras.  In this paper, we construct some interesting quotient bialgebras and quotient Hopf algebras of  the free bialgebra  $\T$ generated by the coalgebra $C:=\K g \oplus \K h$ with $\Delta(g)=g\otimes g$ and $\Delta(h)=1 \otimes h + h \otimes g$. It is clear that the Sweelder algebra $H_4$ and the Taft algebras $T_n$ are among them.  But they are not the objectives in this paper.  In 1977 Radford  used the shuffle type polynomials (see \ref{def:shuffle-polynomials}) to construct several interesting quotient Hopf algberas of $\T$, of which some are infinite dimensional \cite[Proposition 4.7]{R1977} and some are finite dimensional.  In \cite{JZ2021} the authors used the Lyndon-Shirshov basis  \cite{CFL1958,S1958} to study the shuffle type polynomials, and obtained the non-commutative version of the binomial theorem.   

In this paper,  we  make use of  the Lyndon-Shirshov basis and the shuffle polynomials further to study the quotient Hopf algebras of $\T$.  We first show  that the so-called non-commutative Fa\`a di Bruno bialgebra $\BFdBnc$  is contained in the bialgebra $\T$, see Theorem \ref{thm:L-cong-FdB}. The free non-commutative  Fa\`a di Bruno Hopf algbera  was constructed by Brouder--Frabetti--Krattenthaler \cite{BFK2006} from the noncommutative formal diffeomorphisms, while the (commutative) Fa\`a di Bruno Hopf algebra was introduced in Doubilet \cite{D1974} and Joni--Rota \cite{JR1979} from the set partitions, as an important example of incidence Hopf algebras.  By considering  the Lyndon words of the form $\omega_{r}=g\underbrace{h\cdots h}_{r}$ for $r\ge 0$  and the  Lyndon-Shirshov basis  $\{E_{\omega_{r}}|~r\ge 0\}$, we found  a quotient bialgebra $\overline{\T}$  of $\T$, which is an Ore-extension of the Fa\`a di Bruno bialgebra $\BFdB$, see Theorem \ref{thm:Tt-Ore-extension-of-FdB}.   Moreover, an interesting phenomenon  occurs in $\overline{\T}$:
the image $\pi(\BFdBnc)$ of  the non-commutative Fa\`a di Bruno bialgebra $\BFdBnc$ in the bialgebra $\overline{\T}$ generated by the Lyndon words $\{\omega_r\}_{r\ge 0}$ has a PBW basis, which can be obtained by affixing the generators $\{\omega_r\}_{r\ge 1}$   to the PBW basis of $\BFdB$.  The two bialgebras $\pi(\BFdBnc)$ and $\BFdB$ intersect trivially  and generate a sub-bialgebra  $\mathcal{B}_F$ in $\Tt$ with a combined PBW basis from those of the two Fa\`a di Bruno bialgebras, see Theorem \ref{FdBnc+FdB}. Thus,  it is more reasonable to  view $\pi(\BFdBnc)$, instead of $\BFdBnc$,  as a suitable non-commutative version of the commutative Fa\`a di Bruno bialgebra $\BFdB$, see Remark \ref{rm2.26}.   

In Section \ref{Section 3} and Section \ref{Section 4}, we go further to investigate the quotient Hopf algebras of the bialgebra $\overline{\T}$ which have good homological properties.  Let  $\Char\K=p$, $p|n$, $\mathcal{E}_{gh} \neq 0$, $1\le j\le p-2$ and $d_{1}\le d_{2}\le \ldots \le d_{p-2}$.  We construct the following chain of quotient Hopf algebras of $\Tt$:
\begin{align*}
\Tt \stackrel{}\twoheadrightarrow \Tt_{n}\twoheadrightarrow \Tt_{n}'(p)\twoheadrightarrow \Tt_{n}(p) \twoheadrightarrow \Tt_{n}(p;d_{j},d_{j-1},\ldots,d_{1})\twoheadrightarrow \ldots \twoheadrightarrow \Tt_{n}(p;d_{p-2},d_{p-3},\ldots,d_{1}).
\end{align*}
They have the following properties:
\begin{itemize}
\item [(a)] $\Tt_{n}$ has subexponential growth (of infinite GK-dimension),
see Lemma~\ref{lem:Ttn-PBW};
\item [(b)] $\Tt_{n}'(p)$ is of GK-dimension $p-1$, see Lemma~\ref{lem:Ttnp'-PBW};
\item  [(c)] $\Tt_{n}(p)$ is of GK-dimension $p-2$, see Lemma~\ref{lem:Ttnp-PBW};
\item [(d)] $\Tt_{n}(p;d_{j},d_{j-1},\ldots d_{1})$ is of GK-dimension $p-2-j$, see Theorem~\ref{thm:Ttnpd-PBW};
\item [(e)] Let $j=p-2$ in part (d). Then $\Tt_{n}(p;d_{p-2},d_{p-3},\ldots, d_{1})$ is of dimension $np^{(1+\sum_{i=1}^{p-2}d_{i})}$, see Remark \ref{rmk:Ttnpd-f.d.} (i).
\end{itemize}
The Hopf algebras in $(b)$-$(e)$ are AS-Gorenstein, see Corollary \ref{cor:examples-AS-Gorenstein}.  Examples of AS-regular algebras can be found in Corollary \ref{cor:examples-AS-regular}.

In Section \ref{Section 5}, we study the coradical filtration and the related structures of the Hopf algebra $\Tt_{\pm 1}$, the Fa\`a di Bruno Hopf algebra $\HFdB$ and the Hopf algebra $\Tt_{p}(p;1,\ldots,1)$. 
As a byproduct, we show that  $\gr(\Tt_{\pm 1}^{cop})$ with $\Char\K=0$ and $\gr(\Tt_{p}(p;1,\ldots,1)^{cop})$ with $\Char\K=p$ both are the bosonizations of the Nichols algebras  $\BN(V_{1,2})$, where the braiding vector space $V_{1,2}$ appeared in \cite[Chpater 4]{AAH12021} and \cite[Section 4]{AAH22021}.

\section{Preliminaries}\label{Section 1}

Let $\K$ be a field. 
For a colagebra $C$ over $\K$, denote by $\G(C)$  the set of all grouplike elements of $C$, that is, $\G(C):=\{g\in\ C\mid \Delta(g)=g\otimes g,~\epsilon(g)=1\}$, and by $\Pp_{g,g'}(C)$ for $g,g'\in \G(C)$  the set of $(g,g')$-\textit{shew primitive} elements, that is, $\Pp_{g,g'}(C):=\{h\in C\mid \Delta_{C}(h)=g\otimes h+h\otimes g',\epsilon_{C}(h)=0\}$. In particular, the elements of $\Pp_{1,1}(C)$ are said to be \textit{primitive}. Write $\Pp(C):=\bigcup_{g,g'\in \G(C)}\Pp_{g,g'}(C)$.  The \textit{coradical} $\corad(C)$ of $C$ is the sum of all simple subcoalgebras of $C$. The coalgebra $C$ is said to be \textit{pointed} or  \textit{connected} if corad$(C)=\K\G(C)$ or $\corad(C)=\K$, respectively.

\subsection{Lyndon words, Lyndon-Shirhsov basis and shuffle type polynomials}\qquad

 First we recall some results of Lyndon words and Lyndon-Shirshov basis. We follow the definitions of \cite{L1997}. Suppose that $(X,\preceq)$ is a totally ordered set, and
$\langle X \rangle$ is the free monoid generated by $X$ with the lexicographic order induced by $\preceq$, see \cite[p.64]{L1997}. The empty word is denoted by $1_{X}$. In particular, write $\alpha\prec \beta$ if $\alpha\preceq \beta$, but $\alpha\neq \beta$ for $\alpha,\beta\in \langle X \rangle$. If $\alpha=x_{1}x_{2}\cdots x_{n},~ x_{i}\in X$, the \textit{length} of $\alpha$ is $n$ and is denoted by $|\alpha|$. For $y\in X$, denote by $\alpha(y):=\#\{x_{i}|~x_{i}=y,~1\le i\le n\}$. $\K\langle X\rangle$ is the free associative algebra generated by $X$.

\begin{definition}{\cite[p.64]{L1997}}\label{def:Lyndon-words}
A word $\alpha$ in $\langle X \rangle\setminus \{1\}$ is a \textit{Lyndon word} (or \textit{Lyndon-Shirshov word}) if $\alpha\in X$, or $\alpha=\beta\gamma$ such that $\alpha\prec \gamma\beta$ for $\beta,\gamma\in \langle X \rangle \setminus \{1\}$. The set of all Lyndon words is denoted by $\mathcal{L}$.
\end{definition}

\begin{example}
Let $(X,\preceq)=(\{g,h\},g\prec h)$. The list of Lyndon words is:
\begin{align*}
\mathcal{L}:=\{g,h,gh,ggh,ghh,gggh,gghh,ghhh,ggggh,ggghh,gghgh,gghhh,ghghh,ghhhh,\ldots\}.
\end{align*}
\end{example}

The factorization of a Lyndon word is not unique, see \cite[Propositions 5.1.2, 5.1.3]{L1997}. So we usually consider its standard factorization.  For a composed word $\alpha=\beta\gamma\in \mathcal{L}\setminus X$, $\beta,\gamma\in \mathcal{L}$,  the pair $(\beta,\gamma)$  is called the \textit{standard factorizaion} of $\alpha$ if $\gamma$ is of maximal length. Denote it by $\operatorname{st}(\alpha)$ or $(\alpha_{L},\alpha_{R})$.

\begin{lemma}{\cite[Theorem 5.3.1]{L1997}}\label{lem:Lyndon-Shirshov-basis}
Let $\mathfrak{L}(X)$ be the free Lie algebra over $\K$ generated by $X$. Define a map $E:\mathcal{L}\longrightarrow \mathfrak{L}(X)$ inductively by $E(x)=x$ if $x\in X$ and
\begin{align*}
E(\alpha)=[E(\beta),E(\gamma)]:=E(\beta)E(\gamma)-E(\gamma)E(\beta),
\end{align*}
if $\alpha\in \mathcal{L}\setminus X$ and $\operatorname{st}(\alpha)=(\beta,\gamma)$. Then $\mathfrak{L}(X)$ is a free $\K$-module with $E(\mathcal{L})$ as a basis, called \textit{Lyndon-Shirshov basis}. We  write  $E_{\alpha}$ for $E(\alpha)$  for  the convenience in the computations. 
\end{lemma}

\begin{example}
Let $(X,\preceq)=(\{g,h\},g\prec h)$. Then 
\begin{align*}
E_{g} &=g, \qquad E_{h}=h, \\
E_{gh} &=[E_{g},E_{h}]=gh-hg,\\
E_{ggh}&=[E_{g},E_{gh}]=[g,[g,h]], \qquad E_{ghh}=[E_{gh},E_{h}]=[[g,h],h],\\
E_{gggh}&=[E_{g},E_{ggh}]=[g,[g,[g,h]]], \qquad E_{gghh}=[E_{g},E_{ghh}]=[g,[[g,h],h]],\\
E_{ghhh}& =[E_{ghh},E_{h}]=[[[g,h],h],h],\\
E_{ggggh} &= [E_{g},E_{gggh}]=[g,[g,[g,[g,h]]]], \qquad E_{ggghh} = [E_{g},E_{gghh}]=[g,[g,[[g,h],h]]],\\
E_{gghgh} &=[E_{ggh},E_{gh}]=[[g,[g,h]],[g,h]], \qquad E_{gghhh}=[E_{g},E_{ghhh}]=[g,[[[g,h],h],h]],\\
E_{ghghh} &=[E_{gh},E_{ghh}]=[[g,h],[[g,h],h]], \qquad E_{ghhhh}=[E_{ghhh},E_{h}]=[[[[g,h],h],h],h].
\end{align*}
Let $\omega_{r}:=g\underbrace{h\cdots h}_{r}$ for $r\ge 1$. Then
\begin{align*}
E_{\omega_{r}}=[E_{\omega_{r-1}},E_{h}]=[[[[g,\underbrace{h],h],\ldots,h}_{r}].
\end{align*}
\end{example}

\begin{remark}\label{rmk:PBW-of-free-algebra}
By \cite[Corollary 5.3.9]{L1997}, $\K\langle X\rangle$ is a free $\K$-module with a $\operatorname{PBW}$ basis
$\{E_{\alpha_{1}}^{n_{\alpha_{1}}}E_{\alpha_{2}}^{n_{\alpha_{2}}}\cdots E_{\alpha_{m}}^{n_{\alpha_{m}}}|\\~\alpha_{1}\succ \alpha_{2}\succ \ldots \succ\alpha_{m}, ~\alpha_{i}\in \mathcal{L},~n_{\alpha_{i}}\ge 0,m\ge 1\}$.
In order to use the Diamond Lemma \cite{B1978}, we define a reduction order $\prec_{red}$ for the Lyndon-Shirshov basis: for $\alpha,\beta\in \mathcal{L}$, $E_{\beta}\prec_{red} E_{\alpha}$ if $\beta\succ \alpha$.
\end{remark}

\begin{lemma}{\cite[Proposition 5.1.4]{L1997}}\label{lem:condition-alpha-zeta-Lyndon-word}
Let $\alpha\in \mathcal{L}$ and $\operatorname{st}(\alpha)=(\alpha_{L},\alpha_{R})$. Then for any $\zeta\in \mathcal{L}$ such that $\alpha\prec \zeta$, the pair $(\alpha,\zeta)=\operatorname{st}(\alpha\zeta)$ if and only if $\alpha_{R}\succeq \zeta$.
\end{lemma}

The following commutator lemma concerning Lyndon-Shirshov basis is the combination of Lothaire {\cite[Lemma 5.5.3]{L1997}}, Kharchenko {\cite[Lemma 6]{K1999}} and Rosso {\cite[Theorem 1]{R1999}}.

\begin{lemma}\label{lem:commutators}
Let $(X,\preceq)$ be a totally odered set, $\mathcal{L}$ be the set of Lyndon words in $\langle X \rangle$, and $\alpha,\beta\in \mathcal{L}$, $\alpha\prec \beta$, and $\operatorname{st}(\alpha)=(\alpha_{L},\alpha_{R})$ if $\alpha\in \mathcal{L}\setminus X$. Then
\begin{itemize}
\item [(a)] If $\alpha\in X$ or $\alpha_{R}\succeq \beta$, then $[E_{\alpha}, E_{\beta}]=E_{\alpha\beta}$;
\item [(b)] If $\alpha_{R} \prec \beta$,
then there is a set $\Gamma \subseteq \mathcal{L}$ such that $\alpha\beta\in \Gamma$ and
\begin{align}
[E_{\alpha}, E_{\beta}]=\sum_{\gamma\in \Gamma}k_{\gamma}E_{\gamma},
\end{align}
where $k_{\alpha\beta}\neq 0$, for every $\gamma\in \Gamma$, $\gamma$ satisfies: $k_{\gamma}\in \Z$, $\alpha\beta\preceq \gamma\prec \beta$ and $ \gamma(x)=\alpha(x)+\beta(x)~for~all~x\in X$.
\end{itemize}
\end{lemma}

\begin{example}
Let $(X,\preceq)=(\{g,h\},g\prec h)$. 
Note that $\st(ggh)=(g,gh)$, $\st(gggh)=(g,ggh)$ and $ggh\prec gh\prec h$. Thus 
\begin{align*}
[E_{ggh},E_{h}] &=[[E_{g},E_{gh}],E_{h}]=[[E_{g},E_{h}],E_{gh}]+[E_{g},[E_{gh},E_{h}]]=[E_{g},E_{ghh}]=E_{gghh},\\
[E_{gggh},E_{h}] &=[[E_{g},E_{ggh}],E_{h}]=[[E_{g},E_{h}],E_{ggh}]+[E_{g},[E_{ggh},E_{h}]]\\
&=-[E_{ggh},E_{gh}]+[E_{g},E_{gghh}]=E_{ggghh}-E_{gghgh},\\
[E_{gghh},E_{h}] &=[[E_{g},E_{ghh}],E_{h}]=[[E_{g},E_{h}],E_{ghh}]+[E_{g},[E_{ghh},E_{h}]]\\
&=[E_{gh},E_{ghh}]+[E_{g},E_{ghhh}]=E_{ghghh}+E_{gghhh}.
\end{align*}
\end{example}
For more details we refer to \cite[Section 1.2]{JZ2021}.

\begin{definition}\label{def:shuffle-polynomials}
Let $X=\{g,h\}$.
Then in $\K\langle X \rangle$,
\begin{align*}
(g+h)^n=\sum_{i=0}^{n}\mathcal{SH}_{k,n-k}(h,g),
\end{align*}
where $\mathcal{SH}_{k,n-k}(h,g)$ is the sum of words $\alpha$ in $(g+h)^{n}$ such that $\alpha(h)=k$ and $\alpha(g)=n-k$, that is, $\mathcal{SH}_{k,n-k}(h,g)=h^{k}\shuffle g^{n-k}$, where $\shuffle$ is the shuffle product (see \cite{R1979}). We call it the $(k,n-k)$-\textit{shuffle type polynomials}.
\end{definition}

\begin{remark}
(i) By definition, $\mathcal{SH}_{i,j}(h,g)$ has the recursion (see also \cite[Section 4.2]{R1977}):
\begin{align}
\label{formula 2}
\mathcal{SH}_{i,j}(h,g)&=h\mathcal{SH}_{i-1,j}(h,g)+g\mathcal{SH}_{i,j-1}(h,g),\\
\label{formula 3}
\mathcal{SH}_{i,j}(h,g)&=\mathcal{SH}_{i-1,j}(h,g)h+\mathcal{SH}_{i,j-1}(h,g)g.
\end{align}
(ii) From Remark \ref{rmk:PBW-of-free-algebra}, the shuffle type polynomials can be written as  linear combinations of products of the Lyndon-Shirshov basis in the lexicographic order, see \cite[Section 2]{JZ2021} for the details.
\end{remark}

\subsection{Ore extension and GK-dimension}\qquad

 We follow the definitions of \cite{KL2000,MR2001,BOZ2015}.
\begin{definition}{\cite[Chapter 1.2.3, p.15]{MR2001}}
Let $S$ be a ring and $R$ a subring of $S$. Suppose that
\begin{enumerate}
\item [(a)] $\sigma$ is a endmorphism of $R$;
\item [(b)]$\delta$ is a $\sigma$-derivation of $R$, that is, for all $a,b\in R$,
\begin{align*}
\delta(ab) = \delta(a)\sigma(b) + a\delta(b);
\end{align*}
\item [(c)] $x\in S\setminus R$.
\end{enumerate}
A ring extension $R\subseteq S$ is called an \textit{Ore extension} of $R$ if $S=R[x;\sigma,\delta]$  freely generated over $R$ by an element $x$ subject to the relation:
\begin{align*}
ax=x\sigma(a)+\delta(a), \qquad \forall a\in R,
\end{align*}
such that $R[x;\sigma,\delta]=\{\sum_{i}x^{i}a_{i}|~\{a_{i}\}\subseteq R ~\text{has finite support}\}$. In case $\sigma=id$ and $\delta\neq 0$, then $R[x;id,\delta]$ is written as $R[x;\delta]$; if $\delta=0$, then written as $R[x;\sigma]$.
\end{definition}

\begin{definition}{\cite[Section 2.1]{BOZ2015}}
Let $H$ be a Hopf algebra over $\K$ and $R$ a Hopf subalgebra of $H$. $H$ is called a \textit{Hopf Ore extension} (or \textit{HOE}) of $R$ if $H=R[x;\sigma,\delta]$, $\sigma$ is an automorphism and there exist $a,b\in R$ and $v,w\in R\otimes R$ such that
\begin{align*}
\Delta_{H}(x)=a\otimes x+x\otimes b+v(x\otimes x)+w.
\end{align*}
\end{definition}

\begin{definition}{\cite{KL2000}}
Let $A$ be a finitely generated algebra over $\K$. Suppose that $V$ is a finite dimensional generating subspace of $A$, that is,
$$
A=\bigcup_{n \in \mathbb{N}} V_{n},\qquad \text { where } V_{n}=\K+V+V^{2}+\ldots+V^{n}.
$$
Denote by $d_{V}(n)$ the function $\text{dim}_{\K}(V_{n})$. The \textit{Gelfand-Kirillov dimension} (or \textit{GK-dimension}) of $A$ is defined as
\begin{align*}
\text{GKdim}(A)=\overline{\lim_{n\to \infty}} \text{~log}_{n}(d_{V}(n)).
\end{align*}
If a $\K$-algebra $B$ is not finitely generated, then we define the GK-dimension of $B$ as
\begin{align*}
\operatorname{GKdim}(B)=\sup_{A}\{\operatorname{GKdim}(A) \mid A \subseteq B, A \text { finitely generated }\}.
\end{align*}
\end{definition}

Note that the GK-dimension of a finitely generated algebra $A$ is independent of the particular choice of $V$, see \cite[Lemma 1]{KL2000}. So we usually choose $V$ such that $1\in V$. Then $V_{n}=V^{n}$.

\begin{lemma}{\cite[Lemma 3.5]{KL2000}}\label{lem:GKdim-Ore-extension}
Let $A$ be a $\K$-algebra with a derivation $\delta$ such that each finite dimensional subspace of $A$ is contained in a $\delta$-stable finitely generated subalgebra of $A$. Then $\operatorname{GKdim}(A[x;\delta])=\operatorname{GKdim}(A)+1$.
\end{lemma}

\begin{lemma}{\cite[Lemma 4.3]{KL2000}}\label{lem:GKdim=}
Let $A$ be a commutative $\K$-algebra with a subalgebra $B$ such that $A$ is finitely generated as a $B$-module. Then $\operatorname{GKdim}(A)=\operatorname{GKdim}(B)$.
\end{lemma}

\section{The free bialgebra $\T$ and the quotient bialgebra $\Tt$}\label{Section 2}

Let $(X,\prec):=(\{g,h\},g\prec h)$ and denote by $\mathcal{L}$ the set of the Lyndon words in
$\langle X\rangle$.
We study the free bialgebra $\T:=\K\X$ 
and a quotient bialgebra $\Tt$.

\subsection{The free bialgebra $\T$ and the noncommutative Fa\`a di Bruno bialgebra $\BFdBnc$}\label{Section 2.1}\qquad

$\T=\K\langle g,h\rangle$ is a free bialgebra with the counit $\epsilon(g)=1$, $\epsilon(h)=0$ and the coproduct
$\Delta(g)=g\otimes g$ and $\Delta(h)=1\otimes h+h\otimes g$.

Let $\omega_{r}=g\underbrace{h\cdots h}_{r}$ for $r\ge 0$.
It is clear that $\{E_{\alpha}|~\alpha(g)=1\}=\{E_{\omega_{r}}|~r\ge 0\}$.
We now determine the coproduct of $E_{\omega_{r}}$ in $\T$.

\begin{lemma}\cite[4.3]{R1977} For $n>0$,
\begin{align}\label{formula-Radford}
\Delta_{\T}(h^n)=\sum_{k=0}^{n}h^k\otimes \mathcal{SH}_{n-k,k}(h,g).
\end{align}
\end{lemma}

\begin{lemma}\label{lem:coproduct-property}
For $r\ge 2$,  the following coproduct holds in $\T$:
\begin{align*}
\Delta(E_{\omega_{r}})= g\otimes E_{\omega_{r}}+\sum_{\omega_{r}} k_{\alpha,\alpha'} E_{\alpha_{1}}^{a_{1}}E_{\alpha_{2}}^{a_{2}}\cdots E_{\alpha_{m}}^{a_{m}}\otimes E_{\alpha_{1}'}^{a_{1}'}E_{\alpha_{2}'}^{a_{2}'}\cdots E_{\alpha_{n}'}^{a_{n}'}+E_{\omega_{r}}\otimes g^{r+1},
\end{align*}
where $k_{\alpha,\alpha'}\in \Z$, $\alpha_{1}\succ \alpha_{2}\succ \ldots \succ \alpha_{m}$, $\alpha_{1}'\succ \alpha_{2}'\succ \ldots \succ \alpha_{n}'$, $\alpha_{i},\alpha_{j}'\in \mathcal{L}$, $a_{i},a_{j}'\ge 1$ for $1\le i\le m$ and $1\le j\le n$.
More precisely,
\begin{itemize}
\item [(a)] If $\alpha_{1}=h$, then $\alpha_{1}'\neq h$ and there exists some $j$ such that $\alpha_{j}'(g)\ge 2$;
\item [(b)] If $\alpha_{1}\neq h$, then $|\alpha_{1}|\ge 2, \alpha_{1}(g)\ge 1$, $|\alpha_{1}'|\ge 2$ and $\alpha_{1}'(g)\ge 1$.
\end{itemize}
\end{lemma}

\begin{proof}
Note that
\begin{align*}
\Delta(E_{ghh}) =& E_{g}\otimes E_{ghh}+3E_{gh}\otimes E_{gh}E_{g}+(-E_{h}E_{g}+E_{gh})\otimes E_{ggh}+E_{ghh}\otimes E_{g}^{3}.
\end{align*}
Thus Parts (a) and (b) hold for $E_{\omega_{2}}$.  For convenience, we assume $k_{\alpha,\alpha'}=1$. 
By induction on $r$,
\begin{align*}
\Delta(E_{\omega_{r+1}})
=&\Delta([E_{\omega_{r}},h])\\
=&(g\otimes E_{\omega_{r}}+\sum_{\omega_{r}} E_{\alpha_{1}}^{a_{1}}E_{\alpha_{2}}^{a_{2}}\cdots E_{\alpha_{m}}^{a_{m}}\otimes E_{\alpha_{1}'}^{a_{1}'}E_{\alpha_{2}'}^{a_{2}'}\cdots E_{\alpha_{n}'}^{a_{n}'}+E_{\alpha}\otimes g^{r+1})(1\otimes h+h\otimes g)\\
& -(1\otimes h+h\otimes g)(g\otimes E_{\omega_{r}}+\sum_{\omega_{r}}  E_{\alpha_{1}}^{a_{1}}E_{\alpha_{2}}^{a_{2}}\cdots E_{\alpha_{m}}^{a_{m}}\otimes E_{\alpha_{1}'}^{a_{1}'}E_{\alpha_{2}'}^{a_{2}'}\cdots E_{\alpha_{n}'}^{a_{n}'}+E_{\omega_{r}}\otimes g^{r+1})\\
=& g\otimes E_{\omega_{r+1}}+E_{gh}\otimes E_{\omega_{r}}g-hg\otimes E_{g\omega_{r}}+E_{\omega_{r}}\otimes \sum_{k=1}^{r+1}\binom{r+1}{k}E_{\underbrace{g\cdots g}_{k}h}g^{r+1-k}\\
& + E_{\omega_{r+1}}\otimes g^{r+2}+\sum_{\omega_{r}}F_{\text{L}} \otimes (F_{\text{R}}h-hF_{\text{R}}) 
+ \sum_{\omega_{r}} F_{\text{L}}h \otimes F_{\text{R}}g- hF_{\text{L}} \otimes gF_{\text{R}},
\end{align*}
where $F_{\text{L}}:=E_{\alpha_{1}}^{a_{1}}E_{\alpha_{2}}^{a_{2}}\cdots E_{\alpha_{m}}^{a_{m}}$ and $F_{\text{R}}:=E_{\alpha_{1}'}^{a_{1}'}E_{\alpha_{2}'}^{a_{2}'} \cdots E_{\alpha_{n}'}^{a_{n}'}$. To show Part (a) and Part (b) hold for $T_{1}:=F_{\text{L}}\otimes (F_{\text{R}}h-hF_{\text{R}})$ and $T_{2}:=F_{\text{L}}h \otimes F_{\text{R}}g- hF_{\text{L}} \otimes gF_{\text{R}}$, we consider two cases of $F_{\text{L}}\otimes F_{\text{R}}$ in $\Delta(E_{\omega_{r+1}})$:

\textit{Case 1}: $\alpha_{1}=h$ in $F_{\text{L}}\otimes F_{\text{R}}$. By induction on $r+1$, there exists some $j$ such that $\alpha_{j}'(g)\ge 2$ and $\alpha_{1}'\neq h$. By Lemma \ref{lem:commutators} and \cite[Corollary 1.9 (d)]{JZ2021}, there exists a set $\Gamma \subseteq \mathcal{L}$ such that $F_{\text{R}}h= hF_{\text{R}}+\sum_{\gamma\in \Gamma} k_{\gamma} E_{\alpha_{1}'}^{a_{1}'}\cdots E_{\alpha_{\ell}'}^{a_{\ell}'}E_{\gamma}E_{\alpha_{\ell+1}'}^{c_{\ell+1}}\cdots E_{\alpha_{n}'}^{c_{n}}$, where for every $\gamma\in \Gamma$ there exists some $\ell$ such that $\alpha_{\ell}'\succeq \gamma\succ \alpha_{\ell+1}'$, $0\le c_{s}\le a_{s}$, $\ell+1\le s \le n$, $\sum_{s=\ell+1}^{n}(a_{s}'-c_{s})>0$ and $\gamma(x)=h(x)+\sum_{s=\ell+1}^{n}(a_{s}'-c_{s})\alpha_{s}'(x)$ for all $x\in X$. Then $\gamma(g)=\sum_{s=\ell+1}^{n}(a_{s}'-c_{s})\alpha_{s}'(g)$, which means  $\alpha_{j}'(g)\ge 2$ or $\gamma(g)\ge 2$. Thus Part (a) holds for $T_{1}$. Similarly, Part (a) holds for $T_{2}$ from Lemma \ref{lem:commutators} and \cite[Corollary 1.9 (c)-(d)]{JZ2021}. 

\textit{Case 2}: $\alpha_{1}\neq h$ in $F_{\text{L}}\otimes F_{\text{R}}$. The rest is similar to the proof of Case 1.
\end{proof}

Now we show that the noncommutative Fa\`a di Bruno bialgebra is a sub-bialgebra of $\T$. 
Following \cite[Remark 2.1]{BFK2006} and \cite[Section 3.2.6]{ELM2014}, we first recall the Fa\`a di Bruno bialgebra or Hopf algebra, see also \cite{D1974,JR1979}
and \cite[Section 2]{FG2005}.

\begin{definition}{\cite{D1974,JR1979}}\label{def:BFdB}
The \textit{Fa\`a di Bruno bialgebra} $\mathcal{B}_{\operatorname{FdB}}$ as an algebra is defined to be  the polynomial algebra  $\K[u_{1},u_{2},\ldots,u_{n},\ldots]$ in infinitely many variables $\c\{u_{n}\}_{n=1}^{\infty}$, with the counit $\epsilon(u_{n})=\delta_{n,1}$ and the coproduct
\begin{align*}
\Delta(u_{n})=\sum_{k=1}^{n}u_{k}\otimes B_{n,k}(u_{1},u_{2},\ldots,u_{n}),
\end{align*}
where $B_{n,k}(u_{1},u_{2},\ldots,u_{n})$ is the classical partial Bell polynomial, that is,
\begin{align}\label{Bell polynomials}
B_{n,k}(u_{1},u_{2},\ldots,u_{n})=\sum_{k_{1}+k_{2}+\ldots+k_{n}=k \atop k_{1}+2k_{2}+\ldots +nk_{n}=n} \frac{n !}{k_{1}!k_{2}!\ldots k_{n}!}\left(\frac{u_{1}}{1 !}\right)^{k_{1}}\left(\frac{u_{2}}{2 !}\right)^{k_{2}} \ldots\left(\frac{u_{n}}{n !}\right)^{k_{n}}.
\end{align}
The \textit{Fa\`a di Bruno Hopf algebra} is $\mathcal{H}_{\operatorname{FdB}}=\K[u_{1}^{\pm 1},u_{2},\ldots,u_{n},\ldots]$ satisfying the hypothesis of $\mathcal{B}_{\operatorname{FdB}}$ and $\Delta(u_{1}^{-1})=u_{1}^{-1}\otimes u_{1}^{-1}$, see also \cite[Section IX]{JR1979}.
\end{definition}

The Bell polynomials mentioned in above definition were introduced by Bell from the set partitions, see \cite{B1927,B1934}. 
Now we recall the noncommutative Fa\`a di Bruno bialgebra. 

\begin{definition}\cite[Definition 2.4]{BFK2006}.
The \textit{noncommutative Fa\`a di Bruno bialgebra} is defined on the free algebra $\mathcal{B}_{\operatorname{FdB}}^{\operatorname{nc}}=\K\langle a_{0},a_{1},\ldots,a_{n},\ldots\rangle$, with the counit $\epsilon(a_{n})=\delta_{n,0}$ and the coproduct
\begin{align*}
\Delta(a_{n})=\sum_{k=0}^{n}a_{k}\otimes \sum_{i_{0}+i_{1}+\ldots+i_{k}=n-k,\atop i_{0},i_{1},\ldots,i_{k}\ge 0}^{}a_{i_{0}}a_{i_{1}}\cdots a_{i_{k}}.
\end{align*}
\end{definition}

\begin{remark}\label{rmk:BFdBnc-Abelization-cong-BFdB}
$\mathcal{B}_{\operatorname{FdB}}$ (resp. $\mathcal{B}_{\operatorname{FdB}}^{\operatorname{nc}}$) is an $\mathbb{N}$-graded bialgebra, with $\deg(u_{n})=n-1$ for $n\ge 1$ (resp. $\deg(a_{n})=n$ for $n\ge 0$). Thus $\mathcal{B}_{\operatorname{FdB}}$ and $\mathcal{B}_{\operatorname{FdB}}^{\operatorname{nc}}$ are pointed by {\cite[Proposition 4.1.2]{R2012}}. In case $\Char\K=0$, then the abelianization of $\mathcal{B}_{\operatorname{FdB}}^{\operatorname{nc}}$ is $\mathcal{B}_{\operatorname{FdB}}$, that is, there is an isomorphism of bialgebras $$f:\mathcal{B}_{\operatorname{FdB}}^{\operatorname{nc}}/[\mathcal{B}_{\operatorname{FdB}}^{\operatorname{nc}},\mathcal{B}_{\operatorname{FdB}}^{\operatorname{nc}}]\rightarrow \mathcal{B}_{\operatorname{FdB}}, \qquad a_{n}\mapsto \frac{u_{n+1}}{(n+1)!}, \qquad n\ge 0,$$ 
see \cite[Remark 2.12]{BFK2006}.
\end{remark}

Let $L$ be the subalgebra of $\T$ generated by $\{\omega_{n}|~n\ge 0\}$.

\begin{lemma}\label{lem:L-sub-bialgebra}
$L$ is a sub-bialgebra of $\T$.
\end{lemma}
\begin{proof}
By Formula (\ref{formula-Radford}) we have
\begin{align*}
\Delta(\omega_{n})=\sum_{k=0}^{n}\omega_{k}\otimes g\mathcal{SH}_{n-k,k}(h,g),\quad n\ge 0.
\end{align*}
By Formula (\ref{formula 3}), $g\mathcal{SH}_{n-k,k}(h,g)\in L$ for all $0\le k\le n$. Thus $\Delta(L)\subseteq L\otimes L$, as desired
\end{proof}

\begin{theorem}\label{thm:L-cong-FdB}
As bialgebras, $L\cong \mathcal{B}_{\operatorname{FdB}}^{\operatorname{nc}}$. Thus, $\mathcal{B}_{\operatorname{FdB}}^{\operatorname{nc}}$ is a sub-bialgebra of $\T$.
\end{theorem}
\begin{proof}

Note that $\{\omega_{n}|~n\ge 0\}$ is algebraically independent over $\K$. Thus $L$ is a free bialgebra by Lemma \ref{lem:L-sub-bialgebra}. Then, the assignment $\varphi:\omega_{n}\mapsto a_{n}$, $n\ge 0$ gives rise to an isomorphism of algebras $\varphi$ from $L$ to  $\mathcal{B}_{\operatorname{FdB}}^{\operatorname{nc}}$. To show that $\varphi$ is a morphism of coalgebras, it suffices to verify that
\begin{align}\label{formula 12}
g\mathcal{SH}_{n-k,k}(h,g)=\sum_{i_{0}+i_{1}+\ldots+i_{k}=n-k,\atop i_{0},i_{1},\ldots,i_{k}\ge 0}^{}\omega_{i_{0}}\omega_{i_{1}}\cdots \omega_{i_{k}},\quad 0\le k\le n.
\end{align}
It follows from \cite[Theorem 5.9]{JZ2021}, as desired. 
\end{proof}

\begin{remark}
We consider a general construction of Hopf algebras containing $L$ (or $\mathcal{B}_{\operatorname{FdB}}^{\operatorname{nc}}$) as sub-bialgebra, using these Lyndon words $\{\omega_{n}|~n\ge 0\}$. Let $r\ge 1$ and $\T^{[r]}$ the free algebra generated by $\{g,h\}$. It is clear that $\T^{[r]}$ is a bialgebra with coproduct and counit:
\begin{align*}
\Delta(g)=g\otimes g,\qquad \Delta(h)=1\otimes h+h\otimes g^{r},\qquad \epsilon(g)=1, \qquad \epsilon(h)=0.
\end{align*}
From the proof of Lemma \ref{lem:L-sub-bialgebra} and Theorem \ref{thm:L-cong-FdB}, the subalgebra $L^{[r]}$ of $\T^{[r]}$ generated by $\{\omega_{n}|~n\ge 0\}$ is an $\mathbb{N}$-graded, free, pointed bialgebra, with the degree of $\omega_{n}$ defined by $n$. The coproduct is given by the recursion
\begin{align*}
\Delta(\omega_{n})=(g\otimes g)(1\otimes h+h\otimes g^{r})^{n},\quad n\ge 0,
\end{align*}
or equivalently, by Formula (\ref{formula 12}),
\begin{align}\label{formula 13}
\Delta(\omega_{n})=\sum_{k=0}^{n}\omega_{k}\otimes \sum_{i_{0}+\ldots+i_{k}=n-k,\atop i_{0},\ldots,i_{k}\ge 0}^{}\omega_{i_{0}}\omega_{0}^{r-1}\omega_{i_{1}}\cdots \omega_{i_{s}}\omega_{0}^{r-1}\omega_{i_{s+1}}\cdots \omega_{i_{k-1}}\omega_{0}^{r-1}\omega_{i_{k}},\quad n\ge 0.
\end{align}
It is clear that $L^{[1]}=L$. Furthermore, by Formula (\ref{formula 13}), the subalgebra of $L^{[r]}$ generated by $\{\omega_{0}^{r-1}\omega_{n}|~n\ge 0\}$ is a free bialgebra. Thus $L$ (or $\mathcal{B}_{\operatorname{FdB}}^{\operatorname{nc}}$) is isomorphic to a sub-bialgebra of $L^{[r]}$, with the embedding:
\begin{align*}
L\hookrightarrow L^{[r]},~ \omega_{n}\mapsto \omega_{0}^{r-1}\omega_{n}, \quad n\ge 0.
\end{align*}
\end{remark}

\subsection{The bialgebra $\Tt$ and the Fa\`a di Bruno bialgebra $\BFdB$} \qquad 

Let $\Tt:=\T /\mathcal{I}$, where $\mathcal{I}$ is the ideal of $\T$ generated by $\{E_{\alpha}|~\alpha(g)\ge 2,\alpha\in\cL\}$.
We show that $\Tt$ admits a bialgebra structure.

\begin{lemma}\label{lem:equivalent-ideal}
Let $\mathcal{I}_{1}$, $\mathcal{I}_{2}$ be the ideals of $\T$  generated by $\{E_{\omega_{s}}E_{\omega_{t}}-E_{\omega_{t}}E_{\omega_{s}}|~0\le s<t\}$ and $\{gE_{\omega_{n}}-E_{\omega_{n}}g|~n\ge 1\}$ respectively. Then $\mathcal{I}=\mathcal{I}_{1}=\mathcal{I}_{2}$.
\end{lemma}
\begin{proof}
Note that $g=E_{\omega_{0}}$. By Lemma \ref{lem:commutators}, it is easy to see $\mathcal{I}_{2}\subseteq \mathcal{I}_{1} \subseteq \mathcal{I}$.

First we show $\mathcal{I}\subseteq \mathcal{I}_{1}$. Let $E_{\alpha}=[E_{\beta},E_{\gamma}]\in \mathcal{I}$ such that $\operatorname{st}(\alpha)=(\beta,\gamma)$. If $\gamma(g)=0$, then $\gamma=h$. By Lemma \ref{lem:condition-alpha-zeta-Lyndon-word} we have $\beta_{R}\succeq \gamma$. Hence $\beta_{R}=h$. Similarly, for $E_{\beta_{L}}$, since $\beta_{R}(g)=0$ it follows that $(\beta_{L})_{R}=h$. Since  $|\beta|$ is finite, $\beta=\omega_{|\beta|-1}$ by induction on the terms $E_{(\beta_{L})_{L}},\ldots,E_{(\ldots((\beta_{L})_{L})\ldots)_{L}}$. Thus, if $\alpha(g)=2$, then it must have $\beta(g)=\gamma(g)=1$. Hence $E_{\alpha}\in \mathcal{I}_{1}$. If $\alpha(g)>2$, then $\gamma(g)<\alpha(g)$ because $\beta(g)\ge 1$. By induction on $\alpha(g)$, $E_{\gamma}\in \mathcal{I}_{1}$ implies that $E_{\alpha}\in \mathcal{I}_{1}$. Therefore $\mathcal{I}=\mathcal{I}_{1}$.

Next we show $\mathcal{I}_{1}\subseteq \mathcal{I}_{2}$. Let $[E_{\omega_{s}},E_{\omega_{t}}]\in \mathcal{I}_{1}$. If $s=1$, then $[E_{gh}, E_{\omega_{t}}]=[[g,E_{\omega_{t}}],h]-[g,E_{\omega_{t+1}}]\in \mathcal{I}_{2}$. If $s\ge 2$, then $[E_{\omega_{s}},E_{\omega_{t}}]=[[E_{\omega_{s-1}},E_{\omega_{t}}],h]-[E_{\omega_{s-1}},E_{\omega_{t+1}}]$. Thus $[E_{\omega_{s}},E_{\omega_{t}}]\in \mathcal{I}_{2}$ by induction on $s$. Therefore $\mathcal{I}_{1}=\mathcal{I}_{2}$.
\end{proof}

\begin{theorem}\label{thm:Tt-bialgebra}
$\Tt$ is a pointed bialgebra.
\end{theorem}
\begin{proof}
By Lemma \ref{lem:equivalent-ideal} and {\cite[Corollary 5.1.14]{R2012}}, 
it suffices to show that $\mathcal{I}_2$ is a coideal.
It is clear that $\epsilon(\T_2)=0$.
Now we show that $\Delta_{\T}(\mathcal{I}_{2}) \in \T\otimes \mathcal{I}_{2}+\mathcal{I}_{2}\otimes\T$.

Let $[g,E_{\omega_{\ell}}]\in \mathcal{I}_{2}$. If $\ell=1$, then $\Delta([g,E_{gh}])=g^{2}\otimes [g,E_{gh}]+[g,E_{gh}]\otimes g^{3}\in \T\otimes \mathcal{I}_{2}+\mathcal{I}_{2}\otimes \T$. If $\ell\ge 2$, then
\begin{align*}
\Delta([g,E_{\omega_{\ell}}])= &[g\otimes g,g\otimes E_{\omega_{\ell}}+\sum_{\omega_{\ell}} F_{\text{L}}\otimes F_{\text{R}}+E_{\omega_{\ell}}\otimes g^{|\omega_{\ell}|}]\\
=& g^{2}\otimes E_{g\omega_{\ell}}+E_{g\omega_{\ell}}\otimes g^{|\omega_{\ell}|+1}+\sum_{\omega_{\ell}} (gF_{\text{L}}\otimes gF_{\text{R}}-F_{\text{L}}g\otimes F_{\text{R}}g),
\end{align*}
where $F_{\text{L}}:=E_{\alpha_{1}}^{a_{1}}E_{\alpha_{2}}^{a_{2}}\cdots E_{\alpha_{m}}^{a_{m}}$ and $F_{\text{R}}:=E_{\alpha_{1}'}^{a_{1}'}E_{\alpha_{2}'}^{a_{2}'} \cdots E_{\alpha_{n}'}^{a_{n}'}$.
By Lemma \ref{lem:coproduct-property} and \cite[Corollary 1.9 (c)]{JZ2021}, we have
\begin{align*}
gF_{\text{L}}\otimes gF_{\text{R}}-F_{\text{L}}g\otimes F_{\text{R}}g \in
(F_{\text{L}}g+\mathcal{I})\otimes (F_{\text{R}}g+\mathcal{I})-F_{\text{L}}g \otimes F_{\text{R}}g\\
\subseteq
\T\otimes \mathcal{I}+\mathcal{I}\otimes \T=\T\otimes \mathcal{I}_{2}+\mathcal{I}_{2}\otimes \T.
\end{align*}
\end{proof}

 Let $\pi:\T\twoheadrightarrow\Tt$ be the projection.
 Now we determine the coproduct of $\pi(E_{\omega_{k}})$ in  $\Tt$. 

\begin{corollary}\label{cor:coproduct}
The following holds for $k\ge 0$ in $\Tt$:
\begin{align}\label{coproduct formula}
\Delta_{\Tt}(\pi(E_{\omega_{k}})) &=\sum_{r=0}^{k}\pi(E_{\omega_{r}})\otimes \pi(\mathcal{SH}'_{k-r,r+1}),
\end{align}
where $\mathcal{SH}'_{k-r,r+1}$ is the sum of the terms in $\mathcal{SH}_{k-r,r+1}$ (Lyndon-Shirshov basis form) of which the leftmost element is not $E_{h}$, and
\begin{align}\label{SH'-formula}
\pi(\mathcal{SH}_{k-r,r+1}')=\sum_{t_{1}+2t_{2}+\cdots+(k+1)t_{k+1}=k+1 \atop t_{1}+t_{2}+\ldots+t_{k+1}=r+1} \frac{(k+1) !}{t_{1} ! t_{2} ! \cdots t_{k+1} !} \left(\frac{\pi(g)}{1!}\right)^{t_{1}}\left(\frac{\pi(E_{\omega_{1}})}{2!}\right)^{t_{2}}\ldots \left(\frac{\pi(E_{\omega_{k}})}{(k+1)!}\right)^{t_{k+1}}.
\end{align}
\end{corollary}

\begin{proof}
From \cite[Corollary 2.2]{JZ2021} and \cite[Corollary 2.3 (a)]{JZ2021}, $\mathcal{SH}'_{m,1}=E_{\omega_{m}}$ for $m\ge 0$. Then $\Delta(\pi(g))=\pi(g)\otimes \pi(g)=\pi(g)\otimes \pi(\mathcal{SH}'_{0,1})$ and
$\Delta(\pi(E_{gh}))=(\pi\otimes \pi)(g\otimes E_{gh}+E_{gh}\otimes g^{2})=\sum_{r=0}^{1} \pi(E_{\omega_{r}})\otimes \pi(\mathcal{SH}'_{1-r,r+1})$. By Lemma \ref{lem:equivalent-ideal} and the induction on $k$, we have:
\begin{align*}
\Delta(\pi(E_{\omega_{k}}))
=& \Delta(\pi(E_{\omega_{k-1}}))\Delta(\pi(h))-\Delta(\pi(h))\Delta(E_{\omega_{k-1}})
\\
=&\sum_{r=0}^{k-1} \pi(E_{\omega_{r}}) \otimes (\pi(\mathcal{SH}_{k-1-r,r+1}') \pi(h) - \pi(h) \pi(\mathcal{SH}_{k-1-r,r+1}'))\\
& +\pi(E_{\omega_{r}}) \pi(h)\otimes \pi(\mathcal{SH}_{k-1-r,r+1}') \pi(g) - \pi(h)  \pi(E_{\omega_{r}}) \otimes \pi(g) \pi(\mathcal{SH}_{k-1-r,r+1}')\\
=& \pi(g) \otimes \pi([\mathcal{SH}_{k-1,1}',h]) +\sum_{r=1}^{k-1} \pi(E_{\omega_{r}}) \otimes \pi([\mathcal{SH}_{k-1-r,r+1}',h])+\\
& \sum_{r=0}^{k-1} (\pi(h)  \pi(E_{\omega_{r}}) + \pi(E_{\omega_{r+1}})) \otimes \pi(\mathcal{SH}_{k-1-r,r+1}') \pi(g) - \pi(h) \pi(E_{\omega_{r}}) \otimes \pi(\mathcal{SH}_{k-1-r,r+1}') \pi(g)\\
=& \pi(g)\otimes \pi(E_{\omega_{k}}) + \pi(E_{\omega_{k}})\otimes \pi(g)^{k+1}+ \sum_{r=1}^{k-1} \pi(E_{\omega_{r}})\otimes \pi(\left[\mathcal{SH}_{k-1-r,r+1}',h\right]+\mathcal{SH}_{k-r,r}'g)\\
=&\sum_{r=0}^{k}\pi(E_{\omega_{r}})\otimes \pi(\mathcal{SH}_{k-r,r+1}'),
\end{align*}
where $[\mathcal{SH}_{k-1,1}',h]=E_{\omega_{k}}$ since $\mathcal{SH}_{k-1,1}'=E_{\omega_{k-1}}$, and the recursion $\left[\mathcal{SH}_{k-1-r,r+1}',h\right]+\mathcal{SH}_{k-r,r}'g=\mathcal{SH}_{k-r,r+1}'$ follows from \cite[Formula (37)]{JZ2021}. The formula (\ref{SH'-formula}) follows from \cite[Corollary 2.2]{JZ2021} and \cite[Corollary 2.3 (a)]{JZ2021}.
\end{proof}

\begin{remark}\quad \label{rmk:SH'-dual-Bell}
\begin{itemize}
\item[(i)] $\pi(\mathcal{SH}_{k-r,r+1}')$ is the Bell polynomial (\ref{Bell polynomials}).  
Furthermore, $\mathcal{SH}_{k-r,r+1}'$ is the dual Bell differential polynomial $\widehat{B}_{k+1,r+1}^{*}(E_h,E_g)$ introduced by Schimming--Rida \cite{SR1998} for a general noncommutative binomial formula, see \cite[Theorem 4.5 (b)]{JZ2021} for the details.
Therefore, (\ref{coproduct formula}) can be derived from the recursion of the dual Bell polynomials. We just give an approach using the shuffle type polynomials.

\item[(ii)] By the recursion of $\mathcal{SH}_{k-r,r+1}'$ (or $\widehat{B}_{k+1,r+1}^{*}$), Formula (\ref{coproduct formula}) has the recursion:
\begin{align*}
\Delta(\pi(E_{\omega_{k}}))=& (\id\otimes \operatorname{ad}_{r}{\pi(h)}+\operatorname{ad}_{r}{\pi(h)}\otimes \pi(g))^{k}(\pi(g)\otimes \pi(g)),\quad k\ge 0.
\end{align*}
\end{itemize}
\end{remark}

To simplify the notation, we will use $\mathcal{E}_{\omega_k}$ (resp. $h$, $g$) in $\Tt$ instead of $\pi(E_{\omega_k})$ (resp. $\pi(h)$, $\pi(g)$).

\begin{corollary}
Let $\Char\K=p>0$. Then
\begin{align}\label{formula:coproduct-h^p}
\Delta_{\Tt}(h^p)=1\otimes h^{p} + h\otimes \mathcal{E}_{\omega_{p-1}} +h^{p}\otimes g^{p}.
\end{align}
\end{corollary}
\begin{proof}
It follows from Formula (\ref{formula-Radford}) and \cite[Corollary 2.3 (c)]{JZ2021}.
\end{proof}

\begin{theorem}\label{thm:Tt-GK-PBW}
 $\Tt$ has subexponential growth (infinite \textup{GK}-dimension) with a \textup{$\operatorname{PBW}$} basis $$\{h^{n_{2}} \mathcal{E}_{\omega_{k_{1}}}^{n_{\omega_{k_{1}}}} \cdots \mathcal{E}_{\omega_{k_{m}}}^{n_{\omega_{k_{m}}}}|~k_{1}> \ldots > k_{m} \ge 0,~n_{\omega_{k_{i}}},n_{2}\in \mathbb{N},~1\le i\le m\}$$ satisfying the relations:
\begin{align*}
 \mathcal{E}_{\omega_{s}}h = h\mathcal{E}_{\omega_{s}} + \mathcal{E}_{\omega_{s+1}}, \quad \mathcal{E}_{\omega_{s}} \mathcal{E}_{\omega_{t}} = \mathcal{E}_{\omega_{t}} \mathcal{E}_{\omega_{s}},\quad 0\le s<t.
\end{align*}
\end{theorem}

\begin{proof}
By Lemmma \ref{lem:equivalent-ideal}, the relations hold.  By \cite[Section 8]{S1976} (see also \cite[Example 1.3, Section 12.1]{KL2000}), $\Tt$ has subexponential growth.  Applying the Diamond Lemma \cite{B1978} to the order:$\prec_{red}$: $h\prec_{red}\ldots \prec_{red} \mathcal{E}_{\omega_{k}}\prec_{red}\ldots \prec_{red} \mathcal{E}_{gh} \prec_{red} g$, we just need to show that the following overlap ambiguities are resolvable: for all $0\le r<s<t$:
\begin{align}
\label{ambiguties 1}& (\mathcal{E}_{\omega_{r}} \mathcal{E}_{\omega_{s}}) \mathcal{E}_{\omega_{t}} = \mathcal{E}_{\omega_{r}}(\mathcal{E}_{\omega_{s}}\mathcal{E}_{\omega_{t}}),\\
\label{ambiguties 2} & (\mathcal{E}_{\omega_{r}}\mathcal{E}_{\omega_{s}})h=\mathcal{E}_{\omega_{r}}(\mathcal{E}_{\omega_{s}}h).
 \end{align}
It is clear that overlap \eqref{ambiguties 1} is resolvable. By direct computations, the two sides of \eqref{ambiguties 2} can be reduced to the same expression $h\mathcal{E}_{\omega_{s}}\mathcal{E}_{\omega_{r}}+\mathcal{E}_{\omega_{s+1}}\mathcal{E}_{\omega_{r}}+\mathcal{E}_{\omega_{s}}\mathcal{E}_{\omega_{r+1}}$. So it is resolvable.
\end{proof}

\begin{proposition}\label{pro:Tt-N-graded-bialgebra}
$\Tt:=\oplus_{n=0}^{\infty}\Tt(n)$ is an $\bN$-graded pointed  bialgebra, where
\begin{align*}
\Tt(n):=\bigoplus_{r+k_{m}s_{m}+\ldots+k_{1}s_{1}=n \atop k_{m}>\ldots>k_{1}\ge 0} h^{r}\mathcal{E}_{\omega_{k_{m}}}^{s_{m}} \cdots \mathcal{E}_{\omega_{k_{1}}}^{s_{1}}.
\end{align*}
\end{proposition}
\begin{proof}
By Theorem \ref{thm:Tt-bialgebra} it suffices to show that $\Tt$ is $\bN$-graded. Observe that $g\in \Tt(0)$, $h\in \Tt(1)$. Then $\mathcal{E}_{\omega_{k}}\in \Tt(k)$ for all $k\ge 0$. By Theorem \ref{thm:Tt-GK-PBW}, $\Tt(n)\Tt(m)\subseteq \Tt(n+m)$. By Corollary \ref{cor:coproduct}, $\Delta(\Tt(n))\subseteq \sum_{k=0}^n\Tt(k)\otimes\Tt(n-k)$,
 as desired.
\end{proof}

Now we show that $\Tt$ is an Ore extension of the Fa\`a di Bruno bialgebra $\mathcal{B}_{\operatorname{FdB}}$. Let $R$ be the subalgebra of $\Tt$ generated by $\{\mathcal{E}_{\omega_{n}}|~n\ge 0\}$.

\begin{lemma}\quad\label{lem:R-Ore-extension}
 The following hold:
\begin{itemize}
\item[(a)] $R$ is a commutative sub-bialgebra of $\Tt$;
\item[(b)] $\Tt$ is an Ore extension of $R$. Namely,  $\Tt=R[h;\operatorname{ad}_{r}h]$, where
$\operatorname{ad}_{r}h(a):=[a,h]$ for $a\in R$.
\end{itemize}
\end{lemma}

\begin{proof}
By Formula (\ref{coproduct formula}) and Theorem \ref{thm:Tt-GK-PBW}, $R$ is a commutative bialgebra and  $\mathcal{E}_{\omega_{k}}h = h\mathcal{E}_{\omega_{k}}+[\mathcal{E}_{\omega_{k}},h]$. On the other hand, $\operatorname{ad}_{r}h$ is a derivation of $R$ and $\Tt$ is generated by $R$ and $h$. Therefore, $\Tt=R[h;\operatorname{ad}_{r}h]$.
\end{proof}

\begin{lemma}\label{lem:R-maximal-sub-bialgebra}
Assume that $\Char\K=0$. Then $R$ is a maximal commutative subalgebra of $\Tt$, hence a maximal commutative sub-bialgebra of $\Tt$.

\end{lemma}
\begin{proof}
By Lemma \ref{lem:R-Ore-extension} $(a)$, $R$ is commutative sub-bialgebra. Suppose that there is a commutative subalgebra $D$ such that $R\subsetneq D \subsetneq \Tt$. By Lemma \ref{lem:R-Ore-extension} (b),  $d=\sum_{j=0}^{n}h^{j}r_{j}\in D$ for some $r_{j}\in R$, $0\le j\le n$, and $r_{n}\neq 0$. Note that there holds in $\T$:
\begin{align}\label{formula:E_omega-h^j}
[E_{\omega_{i}},h^{j}]=\sum_{k=1}^{j}\binom{j}{k}h^{j-k} E_{\omega_{i+k}}, \quad j\ge 1,~i\ge 0.
\end{align}
Since $[g,r_{0}]=0$ and $r_{n}\neq 0$,  it follows that $[g,d]=[\mathcal{E}_{\omega_{0}},d]=\binom{n}{1}h^{n-1} \mathcal{E}_{\omega_{1}}r_{n}+\sum_{j=0}^{n-2}h^{j}s_{j}\neq 0$, where $s_{0},\ldots,s_{n-2}\in R$. However, $[g,d]=0$ since $D$ is commutative, a contradiction. Consequently, $R$ is a maximal commutative sub-bialgebra of $\Tt$.
\end{proof}

\begin{lemma}\label{lem:R-cong-FdB}
As bialgebras, $R\cong \mathcal{B}_{\operatorname{FdB}}$.
\end{lemma}
\begin{proof}
Note that $R$ and $\mathcal{B}_{\operatorname{FdB}}$ are commutative. Then there is an algebra isomorphism: 
$$\phi:\mathcal{B}_{\operatorname{FdB}}\mapsto R, \qquad u_{n+1}\mapsto \mathcal{E}_{\omega_{n}}, \quad n\ge 0.$$ 
By Corollary \ref{coproduct formula} and Theorem \ref{thm:Tt-GK-PBW}, $\phi$ 
is a bialgebra map. Consequently, $R\cong\mathcal{B}_{\operatorname{FdB}}$ as bialgebras.
\end{proof}

\begin{theorem}\label{thm:Tt-Ore-extension-of-FdB}
$\Tt$ is an Ore extension of the Fa\`a di Bruno bialgebra $\B_{\operatorname{FdB}}$. If $\Char\K=0$, then $\B_{\operatorname{FdB}}$ is a maximal commutative sub-bialgebra of $\Tt$.
\end{theorem}
\begin{proof}
It follows directly from Lemmas \ref{lem:R-Ore-extension}, \ref{lem:R-maximal-sub-bialgebra} and \ref{lem:R-cong-FdB}.
\end{proof}

\begin{remark}\quad\label{rmk:Ttpm}
\begin{itemize}
\item[(i)]$\Tt$ does not admit a Hopf algebra structure. Consider the localization $\Tt_{\pm 1}$ at $g$ with $\Delta(\frac{1}{g})=\frac{1}{g}\otimes \frac{1}{g}$, it is clear that $\Tt_{\pm 1}$ admits a Hopf algebra structure by \cite[Remark 36]{T1971}, see also \cite[Section IX]{JR1979}.
Furthermore, $\Tt_{\pm 1}$ is a HOE of $\mathcal{H}_{\operatorname{FdB}}$.

\item[(ii)] More generally, using the Lyndon-Shirshov basis we can consider a Hopf algebra which contains $\Tt_{\pm 1}$. Let $r\ge 1$ and $\T_{\pm 1}^{[r]}$ be the free algebra generated by $\{g,g^{-1},h\}$ over $\K$. It is clear that $\T_{\pm 1}^{[r]}$ is a bialgebra with the coproduct and counit:
\begin{align*}
\Delta(g^{\pm 1})=g^{\pm 1}\otimes g^{\pm 1},  \qquad \Delta(h)=1\otimes h+h\otimes g^{r}, \qquad \epsilon(g^{\pm 1})=1, \qquad \epsilon(h)=0.
\end{align*}

By Lemma \ref{lem:equivalent-ideal} and the proof of Theorem \ref{thm:Tt-bialgebra}, 
$$\overline{\T_{\pm 1}^{[r]}}:=\T_{\pm 1}^{[r]}/(gg^{-1}-1,~g^{-1}g-1,~E_{\alpha}|~\alpha(g)\ge 2)$$ is an $\mathbb{N}$-graded pointed Hopf algebra, with a \textup{$\operatorname{PBW}$} basis 
$$\{\overline{h}^{n_{2}} \overline{E}_{\omega_{k_{1}}}^{n_{\omega_{k_{1}}}}  \cdots \overline{E}_{\omega_{k_{m}}}^{n_{\omega_{k_{m}}}}g^{n_{1}}|~k_{1}>\ldots>k_{m}\ge 1,~n_{1}\in \mathbb{Z},~n_{2},n_{\omega_{k_{i}}}\in \mathbb{N},~1\le i\le m\}$$ satisfying the relations:
\begin{align*}
\overline{g}~\overline{g}^{-1} = \overline{g}^{-1} \overline{g}=1, \qquad 
\overline{E}_{\omega_{s}} \overline{h} = \overline{h}~\overline{E}_{\omega_{s}} + \overline{E_{\omega_{s+1}}}, \qquad \overline{E_{\omega_{s}}}~\overline{E_{\omega_{t}}}= \overline{E_{\omega_{t}}}~\overline{E_{\omega_{s}}},
\end{align*}
where $0\le s<t$. Similar to Remark \ref{rmk:SH'-dual-Bell} (ii), the coproduct of $\overline{\T_{\pm 1}^{[r]}}$ is given by the recursion:
\begin{align*}
\Delta(\overline{E_{\omega_{k}}})=& (id \otimes \operatorname{ad}_{r}{\overline{h}}+\operatorname{ad}_{r}{\overline{h}\otimes \overline{g}^{r})^{k}(\overline{g}\otimes \overline{g})},\quad k\ge 0.
\end{align*}
Evidently, $\overline{\T_{\pm 1}^{[1]}} \cong \Tt_{\pm 1}$. Furthermore, $\Tt_{\pm 1}$ is isomorphic to a Hopf subalgebra $\overline{\T_{\pm 1}^{[r]}}$, with the embedding:
\begin{align*}
\Tt_{\pm 1}\hookrightarrow \overline{\T_{\pm 1}^{[r]}},~g\mapsto \overline{g}^{r},~h\mapsto \overline{h}, \quad r\ge 0.
\end{align*}
There are some interesting results about $\overline{\T_{\pm 1}^{[r]}}$, but we will discuss them elsewhere.
\end{itemize}
\end{remark}

\subsection{The bialgebra $\mathcal{B}_{F}$ and the Fa\`a di Bruno bialgebras $\BFdB,\BFdBnc$}\qquad 

Recall that $\varphi: \mathcal{B}_{\operatorname{FdB}}^{\operatorname{nc}}\rightarrow L$ is an isomorphism of bialgebras in Theorem \ref{thm:L-cong-FdB} and $\pi: \T \rightarrow \Tt$ is the projection. Then we obtain the following bialgebra map.

\begin{corollary}
$\phi:=\pi \circ \iota \circ \varphi^{-1}:\mathcal{B}_{\operatorname{FdB}}^{\operatorname{nc}} \rightarrow \Tt~(a_{n}\mapsto \pi(\omega_{n}),~n\ge 0)$ is a bialgebra map, where 
$\iota:L\rightarrow \T$ is the inclusion. 
\end{corollary}

It is clear that the bialgebra $\phi(\BFdBnc)=\pi\circ \iota(L)$ is $\mathbb{N}$-graded since $\BFdBnc$ is $\mathbb{N}$-graded (see Remark \ref{rmk:BFdBnc-Abelization-cong-BFdB}) and $\pi$ is an $\mathbb{N}$-graded bialgebra map (see Proposition \ref{pro:Tt-N-graded-bialgebra}). Note that $\phi(\BFdBnc)(0)=\K[\pi(g)]$. It follows from \cite[Proposition 4.1.2]{R2012} that $\phi(\BFdBnc)$ is pointed. Now we study the $\mathbb{N}$-graded pointed bialgebra $\phi(\BFdBnc)$.
First we give the following relations.
\begin{lemma} 
The following hold in $\T$:
\begin{itemize} 
\item [(a)] For  $n\ge 0$,
\begin{align}
\label{formula:omega_n}
\omega_{n} &= \sum_{k=0}^{n} \binom{n}{k} h^{n-k} E_{\omega_{k}},\\
\label{formula:E_omega_m}
E_{\omega_{n}} &= \sum_{k=0}^{n} \binom{n}{k} (-1)^{k} h^{k}\omega_{n-k}. 
\end{align}
\item [(b)] For $m,n\ge 0$,
\begin{align}
\label{formula:E_omega_m-omega_n}
\pi(E_{\omega_{m}})\pi(\omega_{n}) &= \sum_{k=0}^{n} \binom{n}{k} \pi(\omega_{n-k})\pi(E_{\omega_{m+k}}),\\
\label{formula:omega-h}
\omega_{n+1}= \omega_{n}h &= h\omega_{n}+ \sum_{k=0}^{n} \binom{n}{k} h^{n-k} E_{\omega_{1+k}}.
\end{align}
\item [(c)] For  $m,n,r,s\ge 0$,
\begin{align}
\label{formula:omega_m-omega_n}
\omega_{m}\omega_{n} &= \sum_{k=0}^{n} \binom{n}{k} \omega_{m+n-k}E_{\omega_{k}},\\
\label{formula:omega_r-E_omega_s}
\omega_{r}E_{\omega_{s}} &= \sum_{k=0}^{s} \binom{s}{t} (-1)^{t} \omega_{r+t} \omega_{s-t}.
\end{align}
\end{itemize}
\end{lemma}
\begin{proof}
Part (a) follows by induction on $n$. We now show Part (b). Using Formulas (\ref{formula:omega_n}) and (\ref{formula:E_omega-h^j}), we have
\begin{align*}
\pi(E_{\omega_{m}})\pi(\omega_{n}) &= \sum_{k=0}^{n} \binom{n}{k} \pi(E_{\omega_{m}}h^{n-k})\pi(E_{\omega_{k}})\\
 &= \sum_{k=0}^{n} \binom{n}{k} \sum_{t=0}^{n-k} \binom{n-k}{t} \pi(h^{n-k-t}E_{\omega_{m+t}})\pi(E_{\omega_{k}})\\
 &= \sum_{t=0}^{n} \binom{n}{t} \sum_{k=0}^{n-t} \binom{n-t}{k} \pi(h^{n-t-k}) \pi(E_{\omega_{k}}) \pi(E_{\omega_{m+t}})\\
 &= \sum_{t=0}^{n} \binom{n}{t} \pi(\omega_{n-t})  \pi(E_{\omega_{m+t}}).
\end{align*}
By the definition of $\omega_{n}$, we have
$\omega_{n}h=\omega_{n+1}$. It follows from Formula (\ref{formula:E_omega-h^j}) that
\begin{align*}
\omega_{n}h &=hgh^{n}+E_{\omega_{1}}h^{n}\\
&=h\omega_{n} + \sum_{k=0}^{n} \binom{n}{k} h^{n-k}E_{1+k}
\end{align*}
Thus Formula (\ref{formula:omega-h}) holds. It remains to show Part (c). Using Formula (\ref{formula:omega_n}), we obtain
\begin{align*}
\omega_{m}\omega_{n} &=  \sum_{k=0}^{n} \binom{n}{k} \omega_{m} h^{n-k} E_{\omega_{k}}\\
&=  \sum_{k=0}^{n} \binom{n}{k} \omega_{m+n-k} E_{\omega_{k}}. 
\end{align*}
Formula (\ref{formula:omega_r-E_omega_s}) follows directly from Formula (\ref{formula:E_omega_m}).
\end{proof}

Let $\widetilde{\omega}_{n}:=\pi(\omega_{n})$ for $n\ge 0$. Denote by $\mathcal{B}_{F}$ the subalgebra of $\Tt$ generated by $\{\widetilde{\omega}_{n},\mathcal{E}_{\omega_{m}}|~n,m\ge 0\}$. Identifying the Fa\`a di Bruno bialgebra $\B_{\operatorname{FdB}}$ with $R$ in Lemma \ref{lem:R-cong-FdB}, we see that $\mathcal{B}_{F}$ is a bialgebra containing the bialgebras $\pi(\mathcal{B}_{\operatorname{FdB}}^{\operatorname{nc}})$ (here we identify $\BFdBnc$ with $L$) and $\mathcal{B}_{\operatorname{FdB}}$ as sub-bialgebras.

\begin{theorem} \label{FdBnc+FdB}
$\mathcal{B}_{F}$ is an $\bN$-graded pointed bialgebra with a PBW basis
\begin{align*}
\{\widetilde{\omega}_{m}^{d} \mathcal{E}_{\omega_{n_{1}}}^{e_{1}}\cdots \mathcal{E}_{\omega_{n_{j}}}^{e_{j}}|~m,j\ge 1,~n_{1}>\ldots>n_{j}\ge 0,~0\le d\le 1,~e_{1},\ldots,e_{j}\ge 0\}
\end{align*}
such that for $m,n,r,s\ge 0$,
\begin{align}
\mathcal{E}_{\omega_{m}}\mathcal{E}_{\omega_{n}} &=\mathcal{E}_{\omega_{n}} \mathcal{E}_{\omega_{m}} \\
\mathcal{E}_{\omega_{m}}\widetilde{\omega}_{n} &= \sum_{k=0}^{n} \binom{n}{k} \widetilde{\omega}_{n-k} \mathcal{E}_{\omega_{m+k}},\\
\widetilde{\omega}_{m} \widetilde{\omega}_{n} &= \sum_{k=0}^{n} \binom{n}{k} \widetilde{\omega}_{m+n-k} \mathcal{E}_{\omega_{k}},\\
\widetilde{\omega}_{r} \mathcal{E}_{\omega_{s}} &= \sum_{k=0}^{s} \binom{s}{t} (-1)^{t} \widetilde{\omega}_{r+t} \widetilde{\omega}_{s-t}.
\end{align}
Moreover, as the sub-bialgebra of $\mathcal{B}_{F}$, $\pi(\mathcal{B}_{\operatorname{FdB}}^{\operatorname{nc}})$ has a PBW basis:
\begin{align*}
\{1\}\cup \{\widetilde{\omega}_{m} \mathcal{E}_{\omega_{n_{1}}}^{e_{1}}\cdots \mathcal{E}_{\omega_{n_{j}}}^{e_{j}}|~m,j\ge 1,~n_{1}>\ldots>n_{j}\ge 0,~e_{1},\ldots,e_{j}\ge 0\}.
\end{align*}
\end{theorem}

\begin{proof}
It is clear that $\B_{F}$ is an $\bN$-graded bialgebra since $\pi(\mathcal{B}_{\operatorname{FdB}}^{\operatorname{nc}})$ and $\mathcal{B}_{\operatorname{FdB}}$ are $\bN$-graded bialgebras. Observe that $\B_{F}(0)=\K[\pi(g)]$. Thus $\B_{F}$  is pointed from \cite[Proposition 4.1.2]{R2012}.
The relations of $\mathcal{B}_{F}$ follow from the formulas (\ref{formula:E_omega_m-omega_n}), (\ref{formula:omega_m-omega_n}), (\ref{formula:omega_r-E_omega_s}) and Theorem \ref{thm:Tt-GK-PBW}. Now we show that $\mathcal{B}_{F}$ has the PBW basis as claimed. Note that $\pi(g)=\widetilde{\omega}_{0}= \mathcal{E}_{\omega_{0}}$. Applying the Diamond Lemma \cite{B1978}  to the order $\prec_{red}$:  $\ldots \prec_{red} \widetilde{\omega}_{n} \prec_{red} \ldots \prec_{red} \widetilde{\omega}_{1} \prec_{red} \ldots \prec_{red} \mathcal{E}_{\omega_{m}} \prec_{red} \ldots \prec_{red} \mathcal{E}_{\omega_{1}} \prec_{red} \mathcal{E}_{\omega_{0}}$, we just need to show that the following overlap ambiguities for $0\le r<s<t$ are resolvable: 
\begin{align}
\label{overlap-bf-1}
(\mathcal{E}_{\omega_{r}} \mathcal{E}_{\omega_{s}}) \mathcal{E}_{\omega_{t}} &= \mathcal{E}_{\omega_{r}} (\mathcal{E}_{\omega_{s}} \mathcal{E}_{\omega_{t}}), \\
\label{overlap-bf-2}
(\mathcal{E}_{\omega_{r}} \mathcal{E}_{\omega_{s}}) \widetilde{\omega}_{t} &= \mathcal{E}_{\omega_{r}} (\mathcal{E}_{\omega_{s}} \widetilde{\omega}_{t}),\\
\label{overlap-bf-3}
(\mathcal{E}_{\omega_{r}} \widetilde{\omega}_{s}) \widetilde{\omega}_{t} &= \mathcal{E}_{\omega_{r}} (\widetilde{\omega}_{s} \widetilde{\omega}_{t}),\\
\label{overlap-bf-4}
(\widetilde{\omega}_{r} \widetilde{\omega}_{s}) \widetilde{\omega}_{t} &= \widetilde{\omega}_{r} (\widetilde{\omega}_{s} \widetilde{\omega}_{t}).
\end{align}
It is clear that (\ref{overlap-bf-1}) is resolvable. The calculation
\begin{align*}
(\mathcal{E}_{\omega_{r}} \mathcal{E}_{\omega_{s}}) \widetilde{\omega}_{t} 
&= \mathcal{E}_{\omega_{s}} (\mathcal{E}_{\omega_{r}} \widetilde{\omega}_{t})\\
&= \sum_{a=0}^{t} \binom{t}{a} (\mathcal{E}_{\omega_{s}} \widetilde{\omega}_{t-a}) \mathcal{E}_{\omega_{r+a}} \\
&= \sum_{a=0}^{t} \sum_{b=0}^{t-a} \binom{t}{a} \binom{t-a}{b} \widetilde{\omega}_{t-a-b}  (\mathcal{E}_{\omega_{s+b}} \mathcal{E}_{\omega_{r+a}})\\
&= \sum_{a=0}^{t} \sum_{b=0}^{t-a} \binom{t}{a} \binom{t-a}{b} (\widetilde{\omega}_{t-a-b} \mathcal{E}_{\omega_{r+a}}) \mathcal{E}_{\omega_{s+b}} \\
&= \sum_{b=0}^{t}  \binom{t}{b} (\sum_{a=0}^{t-b} \binom{t-b}{a} \widetilde{\omega}_{t-b-a} \mathcal{E}_{\omega_{r+a}}) \mathcal{E}_{\omega_{s+b}} \\
&= \sum_{b=0}^{t}  \binom{t}{b} (\mathcal{E}_{\omega_{r}} \widetilde{\omega}_{t-b}) \mathcal{E}_{\omega_{s+b}}\\
&= \mathcal{E}_{\omega_{r}} (\sum_{b=0}^{t}  \binom{t}{b}  \widetilde{\omega}_{t-b} \mathcal{E}_{\omega_{s+b}}) \\
& = \mathcal{E}_{\omega_{r}} ( \mathcal{E}_{\omega_{s}} \widetilde{\omega}_{t})
\end{align*}
shows that (\ref{overlap-bf-2}) is resolvable. For the left side of (\ref{overlap-bf-3}) we make the calculation
\begin{align*}
(\mathcal{E}_{\omega_{r}} \widetilde{\omega}_{s}) \widetilde{\omega}_{t} 
&= \sum_{a=0}^{s}  \binom{s}{a} \widetilde{\omega}_{s-a} (\mathcal{E}_{\omega_{r+a}}  \widetilde{\omega}_{t})\\
&= \sum_{a=0}^{s} \sum_{b=0}^{t}\binom{s}{a} \binom{t}{b} (\widetilde{\omega}_{s-a} \widetilde{\omega}_{t-b}) \mathcal{E}_{\omega_{r+a+b}}  \\
&= \sum_{a=0}^{s} \sum_{b=0}^{t} \sum_{c=0}^{t-b} \binom{s}{a}  \binom{t}{b} \binom{t-b}{c} \widetilde{\omega}_{s-a+t-b-c} (\mathcal{E}_{\omega_{c}} \mathcal{E}_{\omega_{r+a+b}})  \\
&= \sum_{a=0}^{s} \sum_{c=0}^{t} \sum_{b=0}^{t-c} \binom{s}{a}  \binom{t}{c} \binom{t-c}{b} \widetilde{\omega}_{s-a+t-b-c} \mathcal{E}_{\omega_{r+a+b}}\mathcal{E}_{\omega_{c}}   \\
&= \sum_{c=0}^{t} \binom{t}{c} (\sum_{d=0}^{s+t-c} \sum_{a+b=d} \binom{s}{a}  \binom{t-c}{b} \widetilde{\omega}_{s+t-c-d} \mathcal{E}_{\omega_{r+d}}) \mathcal{E}_{\omega_{c}},
\end{align*}
Using Vandermonde's Identity: $\binom{m+n}{k}=\sum_{i+j=k}\binom{m}{i}\binom{n}{j}$, we see that (\ref{overlap-bf-3}) amounts to
\begin{align*}
& \sum_{c=0}^{t} \binom{t}{c} (\sum_{d=0}^{s+t-c}\binom{s+t-c}{d} \widetilde{\omega}_{s+t-c-d} \mathcal{E}_{\omega_{r+d}}) \mathcal{E}_{\omega_{c}}\\
= & \sum_{c=0}^{t} \binom{t}{c} (\mathcal{E}_{\omega_{r}} \widetilde{\omega}_{s+t-c} ) \mathcal{E}_{\omega_{c}}\\
= & \mathcal{E}_{\omega_{r}} (\sum_{c=0}^{t} \binom{t}{c} \widetilde{\omega}_{s+t-c}  \mathcal{E}_{\omega_{c}})\\
= & \mathcal{E}_{\omega_{r}} ( \mathcal{E}_{\omega_{s}} \widetilde{\omega}_{t} ).
\end{align*}
Thus (\ref{overlap-bf-3}) is resolvable. Similarly, (\ref{overlap-bf-4}) is resolvable from the following calculation and Vandermonde's Identity:
\begin{align*}
(\widetilde{\omega}_{r} \widetilde{\omega}_{s}) \widetilde{\omega}_{t}
&= \sum_{a=0}^{s} \binom{s}{a}  \widetilde{\omega}_{r+s-a} (\mathcal{E}_{\omega_{a}}  \widetilde{\omega}_{t}) \\
&= \sum_{a=0}^{s} \sum_{b=0}^{t} \binom{s}{a} \binom{t}{b}  (\widetilde{\omega}_{r+s-a}  \widetilde{\omega}_{t-b}) \mathcal{E}_{\omega_{a+b}} \\
&= \sum_{a=0}^{s} \sum_{b=0}^{t} \sum_{c=0}^{t-b}  \binom{s}{a} \binom{t}{b} \binom{t-b}{c}   \widetilde{\omega}_{r+s+t-a-b-c} (\mathcal{E}_{\omega_{c}}\mathcal{E}_{\omega_{a+b}}) \\
&= \sum_{a=0}^{s} \sum_{c=0}^{t} \sum_{b=0}^{t-c}  \binom{s}{a} \binom{t}{c} \binom{t-c}{b}   \widetilde{\omega}_{r+s+t-a-b-c} \mathcal{E}_{\omega_{a+b}} \mathcal{E}_{\omega_{c}} \\
&= \sum_{c=0}^{t} \binom{t}{c} (\sum_{a=0}^{s}  \sum_{b=0}^{t-c}  \binom{s}{a}  \binom{t-c}{b}   \widetilde{\omega}_{r+s+t-c-(a+b)} \mathcal{E}_{\omega_{a+b}}) \mathcal{E}_{\omega_{c}} \\
&= \sum_{c=0}^{t} \binom{t}{c} (\sum_{d=0}^{s+t-c} \binom{s+t-c}{d}  \widetilde{\omega}_{r+s+t-c-d} \mathcal{E}_{\omega_{d}}) \mathcal{E}_{\omega_{c}} \\
&= \sum_{c=0}^{t} \binom{t}{c}  ( \widetilde{\omega}_{r}\widetilde{\omega}_{s+t-c}) \mathcal{E}_{\omega_{c}} \\
&=  \widetilde{\omega}_{r} (\sum_{c=0}^{t} \binom{t}{c}  \widetilde{\omega}_{s+t-c} \mathcal{E}_{\omega_{c}}) \\
&= \widetilde{\omega}_{r} (\widetilde{\omega}_{s}\widetilde{\omega}_{t}).
\end{align*}
Similarly, it is not difficult to show that $\pi(\mathcal{B}_{\operatorname{FdB}}^{\operatorname{nc}})$ has the basis as claimed.
\end{proof}

\begin{remark} \label{rm2.26}
From Theorem \ref{FdBnc+FdB} we see that  $\{1\}\cup \{\widetilde{\omega}_{m} \mathcal{E}_{\omega_{n_{1}}}^{e_{1}}\cdots \mathcal{E}_{\omega_{n_{j}}}^{e_{j}}|~m,j\ge 1,~n_{1}>\ldots>n_{j}\ge 0,~e_{1},\ldots,e_{j}\ge 0\}$ forms a PBW basis of $\pi(\BFdBnc)$, while 
$\{\mathcal{E}_{\omega_{n_{1}}}^{e_{1}}\cdots \mathcal{E}_{\omega_{n_{j}}}^{e_{j}}|~m,j\ge 1,~n_{1}>\ldots>n_{j}\ge 0,~e_{1},\ldots,e_{j}\ge 0\}$ is a PBW basis of $\BFdB$.  From the form of the two  PBW bases, we should view $\pi(\BFdBnc)$, instead of $\BFdBnc$, as the non-commutative version of the commutative Fa\'a di Bruno bialgebra $\BFdB$.
\end{remark}

\section{Quotient Hopf algebras of $\Tt$}\label{Section 3}
In this section, we construct some quotient bialgebras $Q$ of $\Tt$, which admit Hopf algebra structures. Let $\pi_{Q}:\T \twoheadrightarrow Q$ be the projection. We still use the notation $\mathcal{E}_{\omega_k}$ (resp. $h$, $g$) in $Q$ instead of $\pi_{Q}(E_{\omega_k})$ (resp. $\pi_{Q}(h)$, $\pi_{Q}(g)$). 

\begin{lemma}{\cite[Corollary 7.6.7]{R2012}}\label{lem:pointed-Hopf-algebra}
Let $A$ be a bialgebra over $\K$ generated by $\mathcal{S}\cup \Pp(A)$, where $\mathcal{S} \subseteq \operatorname{G}(A)$, and the elements of $\mathcal{S}$ are invertible. Then A is a pointed Hopf algebra.
\end{lemma}

\subsection{The Hopf algebra $\Tt_n$}\label{Sec:3.1}\qquad

Let $\Tt_n$ be the quotient algebra $\Tt/\cI_n$, where $\mathcal{I}_{n}$ is the ideal of
$\Tt$ generated by $g^{n}-1$ for $n\ge 2$.

\begin{proposition}\label{pro:Ttn-pointed-Hopf}
$\Tt_n$ is an $\bN$-graded pointed Hopf algebra.
\end{proposition}
\begin{proof}
Since $\Delta(g^n-1)=(g^n-1)\otimes 1+g^n\otimes(g^n-1)$, $\cI_n$ is a bi-ideal. So $\Tt_n$
is a quotient bialgebra of $\Tt$.  By Lemma \ref{lem:pointed-Hopf-algebra}, $\Tt$ is a pointed Hopf algebra. Moreover, $\mathcal{I}_{n}$ is graded. It follows that $\Tt_{n}$ is graded by Propositions \ref{pro:Tt-N-graded-bialgebra}
\end{proof}

\begin{lemma} \quad\label{lem:Ttn-PBW}
\begin{itemize}
\item [(a)] If $\mathcal{E}_{gh}=0$, then $\Tt_n$ is commutative with a
\textup{$\operatorname{PBW}$} basis $\{h^{n_{2}}g^{n_{1}}|~0\le n_{1}\le n-1, n_{2}\in \mathbb{N}\}$.
\item [(b)] If $\mathcal{E}_{gh}\neq 0$, $\Char\K=p>0$ and $p|n$,  then
$\Tt_n$ has infinite \textup{GK}-dimension,
with a \textup{$\operatorname{PBW}$} basis:
\begin{align*}
\{h^{n_{2}}\mathcal{E}_{\omega_{k_{1}}}^{n_{\omega_{k_{1}}}}\mathcal{E}_{\omega_{k_{2}}}^{n_{\omega_{k_{2}}}}
\cdots \mathcal{E}_{\omega_{k_{m}}}^{n_{\omega_{k_{m}}}}g^{n_{1}}|~k_{1}>k_{2}>\ldots
>k_{m}\ge 1,~0\le n_{1}\le n-1,~n_{2},n_{\omega_{k_{i}}}\in \mathbb{N},~1\le i\le m\}.
\end{align*}
Furthermore, $S^{2p}=\id$.
\end{itemize}
\end{lemma}

\begin{proof}
Applying the Diamond Lemma \cite{B1978}
to the order
$\prec_{red}$: $h\prec_{red}\ldots \prec_{red} \mathcal{E}_{\omega_{k}}\prec_{red}\ldots \prec_{red} \mathcal{E}_{gh} \prec_{red} g$,
we just needx to verify the overlaps \eqref{ambiguties 1}, \eqref{ambiguties 2} and the following are resolvable:
\begin{align}
\label{ambiguities 3} & (g^{n-1}g)\mathcal{E}_{\omega_{s}}=g^{n-1}(g\mathcal{E}_{\omega_{s}}),\quad s\ge 0,\\
\label{ambiguities 4} & (g^{n-1}g)h=g^{n-1}(gh).
\end{align}
By direct computations, the overlaps \eqref{ambiguties 1}, \eqref{ambiguties 2} and \eqref{ambiguities 3} are
 resolvable and \eqref{ambiguities 4} amounts to the condition $\binom{n}{1}\mathcal{E}_{gh}g^{n-1}=0$,
 that is, $n\mathcal{E}_{gh}=0$.
 Thus it leads to two cases: 

 (i) $\mathcal{E}_{gh}=0$; (ii) $\mathcal{E}_{gh}\neq 0$, $\Char\K=p$ and $p|n$.

 Now we show $S^{2p}=\id$. Observe that $S(g)=g^{-1}$, $S(h)=-hg^{-1}$, $S^{2}(h)=h+\mathcal{E}_{gh}g^{-1}$, $S(\mathcal{E}_{gh})=-\mathcal{E}_{gh}g^{-3}$ and $g\mathcal{E}_{gh}=\mathcal{E}_{gh}g$. Then by induction on $n$, we obtain
\begin{equation*}
S^{n}(h)=\left\{
\begin{aligned}
& -hg^{-1}-\frac{n-1}{2}\mathcal{E}_{gh}g^{-2} , & \text{ if } n\text{ is odd}, \\
& h+\frac{n}{2}\mathcal{E}_{gh}g^{-1}  , & \text{ if } n \text{ is even}.
\end{aligned}
\right.
\end{equation*}
Therefore, $S^{2p}=\id$ by the fact that $\Tt_{n}$ is generated by $\{g,h\}$ as an algebra.
\end{proof}

\subsection{The Hopf algebras $\Tt_{n}'(p)$, $\Tt_n(p)$ and $\Tt_{\pm 1}'(p)$}\label{Sec:3.2} \qquad

Now we consider some quotient Hopf algebras with finite $\GK$-dimension.
\begin{definition}
Let $\Char\K=p$ and  suppose that $\mathcal{E}_{gh}\neq0$, $p|n$.
Define $\Tt_{n}'(p):=\Tt_{n}/\mathcal{I}_p'$, where
$\mathcal{I}_p'$ is the ideal generated by $\mathcal{E}_{\omega_{p-1}}$.
\end{definition}
\begin{proposition}\label{pro:Ttnp'-pointed-Hopf}
$\Tt_{n}'(p)$ is an $\bN$-graded pointed Hopf algebra.
\end{proposition}
\begin{proof}
By Proposition \ref{pro:Ttn-pointed-Hopf} it suffices to show that $\cI_p'$ is a graded coideal. Indeed,
using Formulas (\ref{coproduct formula}),(\ref{SH'-formula}), we have
\begin{equation}\label{formula:coproduct-E_omega_p-1}
\Delta_{\Tt_{n}'(p)}(\mathcal{E}_{\omega_{p-1}})=
g\otimes \mathcal{E}_{\omega_{p-1}}+\mathcal{E}_{\omega_{p-1}}\otimes g^{p}.
\end{equation}
\end{proof}

\begin{lemma}\label{lem:Ttnp'-PBW}
 $\Tt_{n}'(p)$ has \textup{GK}-dimension $p-1$,
 with a \textup{$\operatorname{PBW}$} basis
 $$\{h^{n_{2}}\mathcal{E}_{\omega_{p-2}}^{n_{\omega_{p-2}}}\mathcal{E}_{\omega_{p-3}}^{n_{\omega_{p-3}}}\cdots
g^{n_{1}}|~0\le n_{1}\le n-1,~n_2,n_{\omega_{k}}\in \mathbb{N},~1\le k\le p-2\}$$
satsifying the relations:
\begin{align*}
& g^{n}=1, \qquad \mathcal{E}_{\omega_{s}}h=h\mathcal{E}_{\omega_{s}}+\mathcal{E}_{\omega_{s+1}}, \\
& \mathcal{E}_{\omega_{p-2}}h=h\mathcal{E}_{\omega_{p-2}},\qquad \mathcal{E}_{\omega_{s}}\mathcal{E}_{\omega_{t}}=\mathcal{E}_{\omega_{t}}\mathcal{E}_{\omega_{s}},
\end{align*}
where $0\le s<t\le p-2$.
\end{lemma}

\begin{proof}
Consider the order $\prec_{red}$: $h\prec_{red}\mathcal{E}_{\omega_{p-2}}\prec_{red} \mathcal{E}_{\omega_{p-3}}\prec_{red}\ldots \prec_{red} g$, the overlaps (\ref{ambiguties 1}), (\ref{ambiguties 2}), (\ref{ambiguities 3}), (\ref{ambiguities 4}) and the following (\ref{ambiguities 5}) are resolvable in $\Tt_{n}(p)$ for all $0\le s\le p-2$:
\begin{align}
\label{ambiguities 5}
& (\mathcal{E}_{\omega_{s}}h)h^{p-1}=\mathcal{E}_{\omega_{s}}(h^{p-1}).
\end{align}
Let $R=\K[g]/(g^{n})$. Since $\Tt_{n}(p)=R[\mathcal{E}_{gh}]\ldots [\mathcal{E}_{\omega_{p-2}}][h;\operatorname{ad}_{r}h]$, the GK-dimension of $\Tt_{n}(p)$ is $p-1$ by Lemma \ref{lem:GKdim-Ore-extension} and Lemma \ref{lem:GKdim=}.
\end{proof}

\begin{remark}\quad \label{rmk:Ttpm'}
\begin{itemize}
\item [(i)] Let $R_{n}'(p)$ be the subalgebra of $\Tt_{n}'(p)$ generated by $\{\mathcal{E}_{\omega_{r}}|~0\le r\le p-2\}$.  Then $\Tt_{n}'(p)=R'_{n}(p)[h;\operatorname{ad}_{r}h]$.
\item [(ii)] From Formula (\ref{formula:coproduct-E_omega_p-1}), Lemma \ref{lem:GKdim-Ore-extension} and Lemma \ref{lem:GKdim=},
 $\Tt_{\pm 1}'(p):=\Tt_{\pm 1}/(\mathcal{E}_{\omega_{p-1}})$ is an $\bN$-graded pointed Hopf algebras of GK-dimension $p$.
\end{itemize}
\end{remark}

\begin{definition}
Let $\Char\K=p$ and  suppose that $\mathcal{E}_{gh}\neq0$, $p|n$.
Define $\Tt_n(p):=\Tt'_n(p)/\cI_p$,
where $\cI_p$ is the ideal of $\Tt'_{n}(p)$ generated by $h^{p}$.
\end{definition}
\begin{proposition}\label{pro:Ttnp-pointed-Hopf}
$\Tt_n(p)$ is an $\bN$-graded pointed Hopf algebra.
\end{proposition}
\begin{proof}
It follows from Lemma \ref{lem:pointed-Hopf-algebra} and the formulas (\ref{formula:coproduct-h^p}), (\ref{formula:coproduct-E_omega_p-1}).
\end{proof}

\begin{lemma}\label{lem:Ttnp-PBW}
 $\Tt_{n}(p)$ has \textup{GK}-dimension $p-2$,
 with a \textup{$\operatorname{PBW}$} basis
 $$\{h^{n_{2}}\mathcal{E}_{\omega_{p-2}}^{n_{\omega_{p-2}}}\mathcal{E}_{\omega_{p-3}}^{n_{\omega_{p-3}}}\cdots
g^{n_{1}}|~0\le n_{2}\le p-1,~0\le n_{1}\le n-1,~n_{\omega_{k}}\in \mathbb{N},~1\le k\le p-2\}$$
satisfying the relations:
\begin{align*}
& g^{n}=1, \qquad h^{p}=0,\qquad \mathcal{E}_{\omega_{s}}h=h\mathcal{E}_{\omega_{s}}+\mathcal{E}_{\omega_{s+1}}, \\
& \mathcal{E}_{\omega_{p-2}}h=h\mathcal{E}_{\omega_{p-2}},\qquad \mathcal{E}_{\omega_{s}}\mathcal{E}_{\omega_{t}}=\mathcal{E}_{\omega_{t}}\mathcal{E}_{\omega_{s}},
\end{align*}
where $0\le s<t\le p-2$.
\end{lemma}

\begin{proof}
Applying the Diamond Lemma to the order $\prec_{red}$: $h\prec_{red}\mathcal{E}_{\omega_{p-2}}\prec_{red} \mathcal{E}_{\omega_{p-3}}\prec_{red}\ldots \prec_{red} g$, the ambiguities (\ref{ambiguties 1}), (\ref{ambiguties 2}), (\ref{ambiguities 3}), (\ref{ambiguities 4}) and following are resolvable in $\Tt_{n}(p)$ for all $0\le s\le p-2$:
\begin{align}
\label{ambiguities 6}
& (hh^{p-1})h=h(h^{p-1}h).
\end{align}
Let $R=\K[g]/(g^{n})$. Since $\Tt_{n}(p)=(R[\mathcal{E}_{gh}]\ldots [\mathcal{E}_{\omega_{p-2}}][h;\operatorname{ad}_{r}h])/(h^{p})$, its GK-dimension is $p-2$ by Lemma \ref{lem:GKdim-Ore-extension} and Lemma \ref{lem:GKdim=}.
\end{proof}

\begin{remark}\label{rmk:Rnp-Ttnp}
Let $R_{n}(p)$ be the subalgebra of $\Tt_{n}(p)$ generated by $\{\mathcal{E}_{\omega_{k}}|~0\le k\le p-2\}$. Then $\Tt_{n}(p)=R_{n}(p)[h;\operatorname{ad}_{r}h]$.
\end{remark}

Recall that Radford \cite[Proposition 4.7]{R1977} introduced the pointed Hopf algebra $\overline{\mathcal{SH}_{n}}:=\T/\mathcal{C}_{n}$, where $\mathcal{C}_{n}$ is the ideal of $\T$ generated by $\{g^{n}-1, \mathcal{SH}_{n-k,k}(h,g)|~0\le k \le n-1\}$, see also \cite[Section 5.3.2]{JZ2021}. We show that the relationship between $\Tt_{n}(p)$ and $\overline{\mathcal{SH}_{n}}$ is as follows.

\begin{corollary}\label{cor:Ttnp-SHn}
Let $n\ge 3$ and $\mathcal{I}'$  be the ideal of $\overline{\mathcal{SH}_{n}}$ generated by $\{\mathcal{E}_{\alpha}\in \overline{\mathcal{SH}_{n}}|~\alpha(1)\ge 2\}$.
\begin{itemize}
\item [(a)] If $\Char\K=0$, then $\overline{\mathcal{SH}_{n}}/\mathcal{I}'\cong \K \mathbb{Z}_{n}$.
\item [(b)] If $\Char\K=p$, then $
\overline{\mathcal{SH}_{p}}/\mathcal{I}'\cong \Tt_{p}(p)$.
\end{itemize}
\end{corollary}

\begin{proof}
From \cite[Corollary 2.2]{JZ2021}, $\mathcal{SH}_{1,n-1}=0$ implies that $n\left((n-1)gh-(n-3)hg\right)=0$ in $\mathcal{SH}_{n}/\mathcal{I}'$. Note that $\mathcal{E}_{ggh}=[g,[g,h]]=0$. If $\Char\K=0$, then $h=0$ and Part (a) holds. Let $\mathcal{K}$ be the ideal of $\T$ generated by $\{E_{\alpha},g^{p}-1,E_{\omega_{p-1}},h^{p}|~\alpha(g)\ge 2\}$. If $\Char\K=p$, then $\mathcal{I}'+\mathcal{C}_{p}=\mathcal{K}$ by \cite[Corollary 2.3 (c)]{JZ2021}. It is clear that $\overline{\mathcal{SH}_{p}}/\mathcal{I}'=\T/\mathcal{C}_{p}/(\mathcal{I}'+\mathcal{C}_{p})/\mathcal{C}_{p}\cong \T/\mathcal{K}$. Therefore, $\overline{\mathcal{SH}_{p}}/\mathcal{I}'\cong \Tt_{p}(p)$ from Theorem \ref{thm:Tt-bialgebra}, Propositions \ref{pro:Ttn-pointed-Hopf}, \ref{pro:Ttnp'-pointed-Hopf}, \ref{pro:Ttnp-pointed-Hopf}.
\end{proof}

\subsection{The Hopf algebra $\Tt_{n}(p;d_{j},d_{j-1},\ldots,d_{1})$} \qquad

Let $\Char\K=p$ and suppose that $\mathcal{E}_{gh}\neq0$, $p|n$, $1\le j\le p-2$,
$d_{j}\ge d_{j-1}\ge \ldots \ge d_{1}\ge 1$. Let $\Tt_{n}(p;d_{j},d_{j-1},\ldots,d_{1}):=\Tt_{n}(p)/\mathcal{I}(d_{j},d_{j-1},\ldots,d_{1})$,
where $\mathcal{I}(d_{j},d_{j-1},\ldots,d_{1})$ is the ideal
of $\Tt_{n}(p)$ generated by $\{\mathcal{E}_{\omega_{k}}^{p^{d_{k}}}|~1\le k\le j\}$.

\begin{proposition}\label{pro:Ttnpd-pointed-Hopf}
$\Tt_{n}(p;d_{j},d_{j-1},\ldots,d_{1})$ is an $\bN$-graded pointed Hopf algebra.
\end{proposition}
\begin{proof}
By Lemma \ref{lem:pointed-Hopf-algebra} it suffices to show that $\mathcal{I}(d_{j},d_{j-1},\ldots,d_{1})$ is a graded coideal. It follows from the formulas (\ref{coproduct formula}) and (\ref{SH'-formula}).
\end{proof}

\begin{theorem}\label{thm:Ttnpd-PBW}
$\Tt_{n}(p;d_{j},d_{j-1},\ldots,d_{1})$ has
 \textup{GK}-dimension $(p-2-j)$,
 with a \textup{$\operatorname{PBW}$} basis
\begin{align*}
\{h^{n_{2}}\mathcal{E}_{\omega_{p-2}}^{n_{\omega_{p-2}}}\mathcal{E}_{\omega_{p-3}}^{n_{\omega_{p-3}}}
\cdots g^{n_{1}}|~ & 0\le n_{2}\le p-1,~0\le n_{\omega_{k}}
\le p^{d_{k}}-1,~1\le k\le j,\\ &  n_{\omega_{k'}}\in \mathbb{N},~j< k'\le p-2,~0
\le n_{1}\le n-1\}
\end{align*}
and the relations:
\begin{align*}
& g^{n}=1, \qquad \mathcal{E}_{\omega_{k}}^{p^{d_{k}}}=0, \qquad h^{p}=0, \\
& \mathcal{E}_{\omega_{s}}h=h\mathcal{E}_{\omega_{s}}+\mathcal{E}_{\omega_{s+1}}, \qquad \mathcal{E}_{\omega_{p-2}}h= h\mathcal{E}_{\omega_{p-2}},\qquad \mathcal{E}_{\omega_{s}}\mathcal{E}_{\omega_{t}} = \mathcal{E}_{\omega_{t}}\mathcal{E}_{\omega_{s}},
\end{align*}
where $1\le k\le j$ and $0\le s<t\le p-2$.
\begin{proof}
Following Lemmas \ref{lem:GKdim-Ore-extension} and \ref{lem:GKdim=}, the GK-dimension of $\Tt_{n}(p;d_{j},d_{j-1},\ldots,d_{1})$ is $p-2-j$. Applying the Diamond Lemma  to the order $\prec_{red}$: $h\prec_{red}\mathcal{E}_{\omega_{p-2}}\prec_{red} \mathcal{E}_{\omega_{p-3}}\prec_{red}\ldots 
\prec_{red} g$, we see that the overlaps (\ref{ambiguties 1}), (\ref{ambiguties 2}), (\ref{ambiguities 3}), (\ref{ambiguities 4}), (\ref{ambiguities 5}), (\ref{ambiguities 6}) and the following are resolvable in $\Tt_{n}(p;d_{j},d_{j-1},\ldots,d_{1})$ for all $i,k,\ell$ such that $0\le i< k\le j$ and $k<\ell\le p-2$:
\begin{align*}
& (\mathcal{E}_{\omega_{i}}\mathcal{E}_{\omega_{k}})\mathcal{E}_{\omega_{k}}^{p^{d_{k}}-1}=\mathcal{E}_{\omega_{i}}(\mathcal{E}_{\omega_{k}}\mathcal{E}_{\omega_{k}}^{p^{d_{k}}-1}),\\
& (\mathcal{E}_{\omega_{k}}^{p^{d_{k}}-1}\mathcal{E}_{\omega_{k}})\mathcal{E}_{\omega_{\ell}}=\mathcal{E}_{\omega_{k}}^{p^{d_{k}}-1}(\mathcal{E}_{\omega_{k}}\mathcal{E}_{\omega_{\ell}}),\\
& (\mathcal{E}_{\omega_{k}}^{p^{d_{k}}-1}\mathcal{E}_{\omega_{k}})h=\mathcal{E}_{\omega_{k}}^{p^{d_{k}}-1}(\mathcal{E}_{\omega_{k}}h),\\
& (\mathcal{E}_{\omega_{k}}\mathcal{E}_{\omega_{k}}^{p^{d_{k}}-1})\mathcal{E}_{\omega_{k}}=\mathcal{E}_{\omega_{k}}(\mathcal{E}_{\omega_{k}}^{p^{d_{k}}-1}\mathcal{E}_{\omega_{k}}).
\end{align*}
\end{proof}
\end{theorem}

\begin{remark}\quad\label{rmk:Ttnpd-f.d.}
\begin{enumerate}
\item [(i)] Let $j=p-2$ in Theorem \ref{thm:Ttnpd-PBW}. Then it derives a class of $np^{\left(1+\sum_{i=1}^{p-2}d_{i}\right)}$-dimensional pointed Hopf algebras $\Tt_{n}(p;d_{p-2},d_{p-3},\ldots,d_{1})$ over $\K$ of charateristic $p$.

\item [(ii)] From Lemmas \ref{lem:Ttn-PBW}, \ref{lem:Ttnp'-PBW}, \ref{lem:Ttnp-PBW} and Theorem \ref{thm:Ttnpd-PBW}, one sees that $\Tt$ has a chain of quotient Hopf algebras:  $\Tt \twoheadrightarrow \Tt_{n}\twoheadrightarrow \Tt_{n}'(p)\twoheadrightarrow \Tt_{n}(p)\twoheadrightarrow 
\ldots \twoheadrightarrow \Tt_{n}(p;d_{j},d_{j-1},\ldots,d_{1})\twoheadrightarrow \ldots \twoheadrightarrow \Tt_{n}(p;d_{p-2},d_{p-3},\ldots,d_{1})$.
\noindent
\end{enumerate}
\end{remark}

\section{Homological properties}\label{Section 4}

In this section, we study  some homological properties of those pointed Hopf algebras constructed in Section \ref{Section 3} with finite GK-dimension.
We begin by introducing some basic knowledge, see e.g. \cite{MR2001,WZ2003,BZ2008} for more details.

Let $R$ be a ring. $R$ is called \textit{affine} if it is finitely generated.
If $R_{R}$ (respectively $_{R}R$) is Noetherian, then $R$ is a \textit{right Noetherian ring} (respectively \textit{left Noetherian ring}). A ring $R$ is a \textit{Noetherian ring} if $R$ is left Noetherian and right Noetherian. 
If there exists some monic polynomial $f\in \mathbb{Z}\langle x_{1},x_{2},\ldots \rangle$ such that $f(r_{1},\ldots,r_{n})=0$ for all $r_{i}\in R$. Then $R$ is said to be a \textit{polynoimal identity ring} (\textit{PI ring} for short).

\begin{lemma}{\cite[Theorem (iv), p.17]{MR2001}}\label{lem:Ore-extention-Noetherian}
Let $S=R[x;\sigma,\delta]$. If $\sigma$ is an automorphism and $R$ is right (or left) Noetherian, then $S$ is right (repectively left) Noetherian.
\end{lemma}

\begin{lemma}{\cite[Corollary (iii), p.481]{MR2001}}\label{lem:PI-ring}
 If $R$ is finitely generated as a right module over a commutative subring $A$, then $R$ is a PI ring.
\end{lemma}

\begin{definition}{\cite[Definition 1.2]{BZ2008}} 
Let $(A,\epsilon)$ be an augmented Noetherian algebra. Then $A$ is \textit{Artin-Schelter Gorenstein} (or \textit{AS-Gorenstein} for short) if
\begin{enumerate}
\item [(a)] injdim$_{A}A=d<\infty$,
\item [(b)] dim$_{\K}\text{Ext}^{d}_{A}(_{A}\K,~_{A}A)=1$ and dim$_{\K}\text{Ext}^{i}_{A}(_{A}\K,~_{A}A)=0$ for all $i\neq d$,
\item [(c)] the right $A$-module versions of (a) and (b) hold.
\end{enumerate}
And if $\text{gldim}~A=d$, then $A$ is called \textit{Artin-Schelter regular} (or \textit{AS-regular}).
\end{definition}

The following lemma is the combination of \cite[Theorem 0.1]{WZ2003} and \cite[Theorem 0.2 (1)]{WZ2003}, see also \cite[Lemma 2.1]{L2020}.

\begin{lemma}{\cite[Theorem 0.1, Theorem 0.2]{WZ2003}}\label{thm:Wu-Zhang}
Every affine noetherian PI Hopf algebra is AS-Gorenstein.
\end{lemma}

We obtain the homological properties of $\Tt_{n}'(p)$, $\Tt_{n}(p)$ and $\Tt_{n}(p;d_{j},d_{j-1},\ldots,d_{1})$ for $1\le j\le p-2$.

\begin{proposition}\label{pro:affine-Noetherian}
$\Tt_{n}'(p)$, $\Tt_{n}(p)$ and $\Tt_{n}(p;d_{j},d_{j-1},\ldots,d_{1})$ are affine and Noetherian.
\begin{proof}

$\Tt_{n}'(p)$ is affine because it is generated by $\{g,h\}$. Note that $R_{n}'(p)=\K[g,\mathcal{E}_{\omega_{1}},\ldots,\mathcal{E}_{\omega_{p-2}}]/(g^{n})$ in Remark \ref{rmk:Ttpm'} (i) is Noetherian. Thus $\Tt_{n}'(p)=R_{n}'(p)[h;\operatorname{ad}_{r}h]$ is Noetherian by Lemma \ref{lem:Ore-extention-Noetherian}. $\Tt_{n}(p)$ and $\Tt_{n}(p;d_{j},d_{j-1},\ldots,d_{1})$ are also affine and Noetherian, as quotients of $\Tt_{n}'(p)$.
\end{proof}
\end{proposition}

\begin{proposition}\label{pro:PI-algebras}
$\Tt_{n}'(p)$, $\Tt_{n}(p)$ and $\Tt_{n}(p;d_{j},d_{j-1},\ldots,d_{1})$ 
are PI algebras.
\end{proposition}

\begin{proof}
By Remark \ref{rmk:Rnp-Ttnp}, $R_{n}(p)$ is a commutative subalgebra of $\Tt_{n}(p)$. Observe that $\Tt_{n}(p)$ is finitely generated as a right $R_{n}(p)$-module because $h^{p}=0$. Thus $\Tt_{n}(p)$ is PI by Lemma \ref{lem:PI-ring}. Similarly, $\Tt_{n}(p;d_{j},d_{j-1},\ldots,d_{1})$ is PI. Note that the subalgebra $R'$ of $\Tt_{n}'(p)$ generated by $\{g,\mathcal{E}_{gh},\ldots,\mathcal{E}_{\omega_{p-2}},h^{p}\}$ is commutative and $\Tt_{n}'(p)$ is a free right $R'$-module with a basis $\{1,h,h^{2},\ldots,h^{p-1}\}$. Therefore, the algebra $\Tt_{n}'(p)$ is PI by Lemma \ref{lem:PI-ring}.
\end{proof}

\begin{corollary}\label{cor:examples-AS-Gorenstein}
$\Tt_{n}'(p)$, $\Tt_{n}(p)$ and $\Tt_{n}(p;d_{j},d_{j-1},\ldots,d_{1})$
are AS-Gorenstein.
\end{corollary}

\begin{proof}
It follows from Proposition \ref{pro:affine-Noetherian}, Proposition \ref{pro:PI-algebras} and Lemma \ref{thm:Wu-Zhang}
\end{proof}

The above examples may be not AS-regular, so we consider $\Tt_{\pm 1}'(p)$ in Remark \ref{rmk:Ttpm'} (ii).

\begin{corollary}\label{cor:examples-AS-regular}
$\Tt_{\pm 1}'(p)$ is AS-regular.
\end{corollary}

\begin{proof}
It is evident that $\Tt_{\pm 1}'(p)$ is affine. Since $\Tt_{\pm 1}'(p)=\K[g^{\pm 1}][\mathcal{E}_{gh}]\ldots [\mathcal{E}_{\omega_{p-2}}][h;\operatorname{ad}_{r}h]$, 
$\Tt_{\pm 1}'(p)$ is Noetherian by Lemma \ref{lem:Ore-extention-Noetherian}. It follows from \cite[Theorem 7.5.3, p.263]{MR2001} that $\Tt_{\pm 1}'(p)$ is of a finite global dimenison. 
Observe that the subalgebra $R'$ of $\Tt_{\pm 1}'(p)$ generated by $\{g^{\pm 1},\mathcal{E}_{gh},\ldots,\mathcal{E}_{\omega_{p-2}},h^{p}\}$ is commutative and $\Tt_{\pm 1}'(p)$ is a free right $R'$-module with a basis $\{1,h,h^{2},\ldots,h^{p-1}\}$. Thus $\Tt_{\pm 1}'(p)$ is PI by Lemma \ref{lem:PI-ring}. Therefore, $\Tt_{\pm 1}'(p)$ is AS-regular from Lemma \ref{thm:Wu-Zhang}.
\end{proof}

\section{The coradical filtrations and related structures}\label{Section 5}
In this section, we study the coradical filtration of $\Tt_{\pm 1}$, $\HFdB$ and $\Tt_{p}(p;\underbrace{1,\ldots,1}_{p-2})$ respectively. If there is no confusion, we set $\Tt_{p}(p;1,\ldots,1):=\Tt_{p}(p;\underbrace{1,\ldots,1}_{p-2})$ for convenience in the sequel.

\subsection{Nichols algebras}\qquad

Let $H$ be a Hopf algebra  with a bijective antipode. Let $^{H}_{H}\mathcal{YD}$ be the category of left Yetter-Drinfel'd modules (see \cite[Definition 11.6.2]{R2012}), whose object $(M,\cdot,\rho)$ is a left $H$-module $(M,\cdot)$  and a left $H$-comodule $(M,\rho)$ (write $\rho(m):=m_{(-1)}\otimes m_{(0)}$), satifying the compatiblity condition:
\begin{align*}
\rho(h\cdot m)=h_{1}m_{(-1)}S_{H}(h_{(3)})\otimes h_{(2)}\cdot m_{(0)}, \quad h\in H,~m\in M.
\end{align*}
$^{H}_{H}\mathcal{YD}$ is a braided monoidal category with the braiding structure: for $M,N\in {^{H}_{H}\mathcal{YD}}$, the braiding 
\begin{align*}
c:M\otimes N\rightarrow N\otimes M, \qquad c(m\otimes n)=m_{(-1)}\cdot n\otimes m_{(0)}, \quad m\in M,~n\in N.
\end{align*}
A Hopf algebra $B$ in $^{H}_{H}\mathcal{YD}$ is called a braided Hopf algebra (see \cite[p. 370]{R2012} and \cite[Section 1]{AS2002}), with the coproduct $\Delta_{B}$ satisfying: $\Delta_{B}(1)=1 \otimes 1$ and
\begin{align*}
\Delta_{B}(bc)=b_{(1)}(b_{(2)(-1)}\cdot c_{(1)}) \otimes b_{(2)(0)}c_{(2)}, \qquad b,c\in B.
\end{align*}

For Subections 5.2 and 5.3, we need the following conclusions.
\begin{lemma}\cite[Theorem 11.7.1]{R2012}\label{lem:Radfordbiproduct}
Let $A$, $H$ be Hopf algebras over $\K$. If there are Hopf algebra maps $\pi: A \rightarrow H$ and $\jmath: H \rightarrow A$  such that $\pi\circ \jmath =id$, then there exists a braided Hopf algebra $(B,\cdot,\rho)\in {_{H}^{H}\mathcal{YD}}$ such that
\begin{align*}
A&\cong B\# H,\\
B &=A^{co~\pi}=\{a\in A|~(id\otimes \pi)(\Delta(a))=a\otimes 1\},\\
\Delta_{B}&=(\Pi\otimes id)\circ \Delta_{A}|B,\\
x\cdot b &=\jmath(x_{(1)})b\jmath(S({x_{(2)}})),\\
\rho(b) &=\pi(b_{(1)})\otimes b_{(2)},
\end{align*}
where $\Pi=id*(\jmath\circ S\circ\pi)$, $x\in H$ and $b\in B$.
\end{lemma}

Recall that
the \textit{coradical filtration} of a coalgebra $C$ over $\K$ is a coalgebra filtration $\{C_{n}\}_{n=0}^{\infty}$ of $C$ such that $C_{0} \subseteq C_{1} \subseteq C_{2} \ldots \subseteq \cup_{n=0}^{\infty} C_{n}$ and $\Delta(C_{n}) \subseteq \sum_{\ell=0}^{n} C_{n-\ell} \otimes C_{\ell}$ , where $C_{0}=\operatorname{corad}{C}$ and  $C_{n}:=C_{n-1} \wedge C_{0}= \Delta^{-1}(C_{n-1} \otimes C + C \otimes C_{0})$ for $n\ge 0$.

Suppose that $\{H_{n}\}_{n=0}^{\infty}$ is the coradical filtration of a Hopf algebra $H$ over $\K$, which is also an algebra filtration. The \textit{graded associated Hopf algebra} $\operatorname{gr}(H)$ of $H$ is an $\mathbb{N}$-graded Hopf algebra such that $\operatorname{gr}(H) = \bigoplus_{n=0}^{\infty} H(n)$, $H(n)=H_{n}/H_{n-1}$, with the multiplication, the coproduct $\Delta_{\operatorname{gr}(H)}$ and the antipode $S_{\operatorname{gr}(H)}$ respectively defined by: let $\overline{a} \in H(m), \overline{b}\in H(n)$. $\overline{a}\overline{b}= \overline{ab}$ is determined by $(a+H_{m-1})(b+H_{n-1}) = ab+H_{m+n-1}$, $a\in H_{m}, b\in H_{n}$;  $S_{\operatorname{gr}(H)}(\overline{a})= \overline{S_{H}(a)}$ is determined by $S_{\operatorname{gr}(H)}(a+H_{m-1})= S_{H}(a)+H_{m-1}$, $a\in H_{m}$ ; and  
\begin{align*}
& \Delta_{\operatorname{gr}(H)}(\overline{a}) = \sum_{\ell=0}^{m} (\pi_{n-\ell} \otimes \pi_{\ell})\circ \Delta_{H}(a), \qquad \pi_{j}: H_{j} \rightarrow H(j)=H_{j}/H_{j-1},~h\mapsto \overline{h}, \qquad a\in H_{m}.
\end{align*}

\begin{definition}{\cite[Definition 2.1]{AS2002}}\label{def:Nichols-algebra}
Let $V\in {^{H}_{H}\mathcal{YD}}$. A braided Hopf algebra $R=\bigoplus_{n=0}^{\infty} R(n) \in {^{H}_{H}\mathcal{YD}}$ is called a \textit{Nichols algebra} if $R(0)\cong \K$, $R(1)\cong V$, $\Pp(R)=R(1)$ and $R$ is generated as an algebra by $R(1)$.
\end{definition}

Note that \cite[Proposition 2.2]{AS2002} shows the existence and the uniqueness of Nichols algebras $R$: 
$R\cong B(V)=T(V)/I(V)$, where $I(V)$ is the sum of $I$ such that $I$ is a homogeneous bi-ideal of $T(V)$ generated by homogeneous elements of degree $\ge 2$.

\subsection{The coradical filtrations of $\Tt_{\pm 1}$ and $\HFdB$}\qquad

We study the coradical filtration of $\Tt_{\pm 1}$
Set $y:=h$, $x:=g$, $x_{0}:=1$ and $x_{m}:=\mathcal{E}_{\omega_{m}}g^{-1}$ for all $m\ge 1$.
Observe that $\K[x^{\pm 1}] \cong \K\mathbb{Z}$.
Then we rewrite the defining relations of $\Tt_{\pm 1}^{cop}$.

\begin{lemma}\label{lem:redefinition-Ttpm}
Let $z_{2}:=x_{2}-\frac{3}{2}x_{1}^{2}$ and $z_{n}:=x_{n}-\frac{(n+1)!}{2^{n}}x_{1}^{n}$ for $n\ge 2$. Then the relations of $\Tt_{\pm 1}^{cop}$ are equivalent to
\begin{align}
\label{relation1-Ttpm}
[x,y] &= x_{1}x,\\
\label{relation2-Ttpm}
[x_{1},y] &= \frac{1}{2} x_{1}^{2}+z_{2},\\
\label{relation3-Ttpm}
[z_{n},y] &= -z_{n}x_{1}-\frac{n(n+1)!}{2^{n}}z_{2}x_{1}^{n-1}+z_{n+1}, \qquad n\ge 2,\\
\label{relation4-Ttpm}
[x,x_{1}] &=0, \qquad [x,z_{m}] =0,\qquad [x_{1},z_{n}]=0, \qquad [z_{n},z_{m}]=0,  \qquad 2\le n< m.
\end{align}
\end{lemma}

\begin{proof}
Rewrite the relations of $\Tt_{\pm 1}^{cop}$ as follows:
\begin{align*}
[x,y]=x_{1}x,\qquad [x_{n},y]=-x_{n}x_{1}+x_{n+1}, \qquad [x,x_{n}]&=0,\qquad [x_{n},x_{m}]=0, \quad 1\le n<m.
\end{align*}
Then the claim holds by induction on $n$.
\end{proof}

Let $\cF_n^{\pm}:=F_{n}^{\pm}\K[x^{\pm 1}]$ for $n \ge 0$, where
\begin{align*}
F_{n}^{\pm}= \bigoplus_{r+(i_{m}-1)s_{m}+\ldots+(i_{2}-1)s_{2}+s_{1}\le n,\atop  i_{m}>\ldots>i_{2}\ge 2} \K  y^{r}z_{i_{m}}^{s_{m}}\cdots z_{i_{2}}^{s_{2}}x_{1}^{s_{1}}
\end{align*}
For example,
\begin{align*}
\cF_{0}^{\pm} &=\K [x^{\pm 1}],\\
\cF_{1}^{\pm} &= (\K 1 \oplus \K y \oplus \K z_{2} \oplus \K x_{1}) \K [x^{\pm 1}],\\
\cF_{2}^{\pm} &= (\K 1 \oplus \K y \oplus \K z_{2} \oplus \K x_{1} \oplus \K z_{3} \oplus \K y^{2} \oplus \K yz_{2} \oplus \K yx_{1} \oplus \K z_{2}^{2} \oplus \K z_{2}x_{1 }\oplus \K x_{1}^{2}) \K [x^{\pm 1}].
\end{align*}
It is clear that $\cF_0^{\pm} \subseteq \cF_1^{\pm} \subseteq \ldots \subseteq \cF_n^{\pm} \subseteq \ldots \subseteq \Tt_{\pm 1}^{cop}$. Following the definition of  $F_n^{\pm}$, we have:

\begin{lemma}\label{lem:Fn-FsFt}
For $n\ge 0$,
$F_{n}^{\pm}\subseteq \sum_{s+t=n\atop s,t<n} F_{s}^{\pm}F_{t}^{\pm}.$
\end{lemma}

\begin{lemma}\label{lem:hopf-filtration}
$\cup_{n=0}^{\infty}\cF_n^{\pm}=\Tt_{\pm 1}^{cop}$. Furthermore, $\{\cF_n^{\pm}\}_{n=0}^{\infty}$ is a Hopf algebra filtration of $\Tt_{\pm 1}^{cop}$.
\end{lemma}
\begin{proof}
Note that $\{y^{r}z_{i_{m}}^{s_{m}}\cdots z_{i_{2}}^{s_{2}}x_{1}^{s_{1}}x^{t}|~i_{m}>\ldots>i_{2}\ge 2,~r,s_{m},\ldots,s_{1}\ge 0,~t\in \mathbb{Z}\}$ is a PBW basis of $\Tt_{\pm 1}^{cop}$. Thus $\cup_{n=0}^{\infty}\cF_n^{\pm}=\Tt_{\pm 1}^{cop}$. It is evident that $\cF_{0}^{\pm} \cF_{n}^{\pm} \subseteq \cF_{n}^{\pm}$ for all $n\ge 0$ from Lemma \ref{lem:redefinition-Ttpm}. To show $\cF_{m}^{\pm} \cF_{n}^{\pm} \subseteq \cF_{m+n}^{\pm}$ for all $m,n\ge 0$, it suffices to show $F_{m}^{\pm} F_{n}^{\pm} \subseteq F_{m+n}^{\pm}$. By Lemma \ref{lem:redefinition-Ttpm}, it is easy to show that $F_{1}^{\pm}F_{n}^{\pm} \subseteq F_{n+1}^{\pm}$
for all $n\ge 0$. Consequently, $F_{2}^{\pm}F_{n}^{\pm} \subseteq (\sum F_{1}^{\pm}F_{1}^{\pm}) F_{n}^{\pm} \subseteq F_{n+2}^{\pm}$.  By induction on $m$, we have $F_{m}^{\pm} F_{n}^{\pm} \subseteq (\sum_{s+t=m\atop s,t<m} F_{s}^{\pm}F_{t}^{\pm}) F_{n}^{\pm} \subseteq F_{m+n}^{\pm}$, using Lemma \ref{lem:Fn-FsFt}. Thus $\{\cF_n^{\pm}\}_{n=0}^{\infty}$ is an algebra filtration. Now we show $\Delta(\cF_{n}^{\pm}) \subseteq \sum_{\ell=0} \cF_{n-\ell}^{\pm} \otimes \cF_{\ell}^{\pm}$. By induction on $n$, we obtain that
\begin{align*}
\Delta(\cF_{n}^{\pm}) &\subseteq \Delta(\sum_{s+t=n \atop s,t<n} \cF_{s}^{\pm}\cF_{t}^{\pm})=\sum_{s+t=n \atop s,t<n} \Delta(\cF_{s}^{\pm})\Delta(\cF_{t}^{\pm})
\subseteq \sum_{s+t=n \atop s,t<n} (\sum_{s_{1}+s_{2}=s} \cF_{s_{1}}^{\pm} \otimes \cF_{s_{2}}^{\pm})(\sum_{t_{1}+t_{2}=t} \cF_{t_{1}}^{\pm} \otimes \cF_{t_{2}}^{\pm}) \\
&\subseteq \sum_{s_{1}+t_{1}+s_{2}+t_{2}=n} \cF_{s_{1}+t_{1}}^{\pm} \otimes \cF_{s_{2}+t_{2}}^{\pm}= \sum_{\ell=0}^{n} \cF_{n-\ell}^{\pm} \otimes \cF_{\ell}^{\pm}.
\end{align*}
Thus $\{\cF_n^{\pm}\}_{n=0}^{\infty}$ is a coalgebra filtration. To show $\{\cF_n^{\pm}\}_{n=0}^{\infty}$  is a Hopf filtration it remains to show $S(\cF_{n}^{\pm}) \subseteq \cF_{n}^{\pm}$. Note that $\sum_{\ell=0}^{n} S(\cF_{n-\ell}^{\pm}) \cF_{\ell}^{\pm} = \K1$. It is clear that $S(\cF_{0}^{\pm}) \subseteq \cF_{0}^{\pm}$. For $n\ge 1$,
by induction on $n$, we obtain:
\begin{align*} 
S(\cF_{n}^{\pm})\cF_{0}^{\pm} &= \K1 + \sum_{\ell=1}^{n} S(\cF_{n-\ell}^{\pm}) \cF_{\ell}^{\pm} \subseteq  \K1+\sum_{\ell=1}^{n} \cF_{n-\ell}^{\pm}\cF_{\ell}^{\pm} \subseteq \cF_{n}^{\pm}.
\end{align*}
\end{proof}

\begin{lemma}\label{lem:first-term-coradicalfil}
Assume $\Char\K=0$. Then $\cF_1^{\pm}$ is the first term of the coradical filtration of $\Tt_{\pm 1}^{cop}$.
\end{lemma}
\begin{proof}
By the Taft-Wilson Theorem, it suffices to show that any non-trivial $(x^N,1)$-skew primitive elements of $\Tt_{\pm 1}^{cop}$ is contained in $F_1^{\pm}$ for $N\ge 0$. It is clear that $\Pp_{x,1}(\Tt_{\pm 1}^{cop})=\K\{1-x,y,x_1\}$.
Let $z$ be a non-trivial $(x^{N},1)$-skew primitive element for $N\ge 2$. Since $z_{2}=x_{2}-\frac{3}{2}x_{1}^{2}\in \Pp_{x^{2},1}(\Tt_{\pm 1}^{cop})$ it is enough to show that $N=2$ and $z=\alpha z_{2}$ for some $0\neq \alpha\in \K$.

Note that $\{y^{r}x_{i_{m}}^{s_{i_{m}}} \cdots x_{i_{1}}^{s_{i_{1}}}x^{s}|~r,s_{i_{m}},\ldots,s_{i_{1}}\ge 0, s\in \mathbb{Z}, i_{m}> \ldots> i_{1}\ge 1\}$ is a basis of $\Tt_{\pm 1}^{cop}$. It follows that
$$z=\sum_{r+Ns_{N}+\ldots+s_{1}=N \atop r, s_{N},\ldots,s_{1}\ge 0}^{} k_{r,s_{N},\ldots,s_{1} } y^{r}x_{N}^{s_{N}} \cdots x_{1}^{s_{1}}.$$
By Formula (\ref{formula-Radford}), we have
\begin{align*}
\Delta(z) =& \sum_{r+Ns_{N}+\ldots+s_{1}=N \atop r, s_{N},\ldots,s_{1}\ge 0}^{} k_{r,s_{N},\ldots,s_{1}} (y^{r} \otimes 1 +  \sum_{k=1}^{n-1} \mathcal{SH}_{n-k,k}(y,x) \otimes y^{k} + x^{r} \otimes y^{r})\\
& \cdot (x_{N}^{s_{N}} \cdots x_{1}^{s_{1}} \otimes 1 + \ldots + x^{N-r} \otimes x_{N}^{s_{N}}\cdots x_{1}^{s_{1}}).
\end{align*}
Thus, if $r\neq 0$, then we obtain
\begin{align*}
\sum_{r+Ns_{N}+\ldots+s_{1}=N \atop r, s_{N},\ldots,s_{1}\ge 0}^{} k_{r,s_{N},\ldots,s_{1}} x^{r}x_{N}^{s_{N}} \cdots x_{1}^{s_{1}}=0.
\end{align*}

Thus $k_{r,s_{N},\ldots,s_{1}}=0$ for all $s_{N},\ldots,s_{1}$ and so $z=\sum_{Ns_{N}+\ldots+s_{1}=N}^{} \ell_{s_{N},\ldots,s_{1}} x_{N}^{s_{N}} \cdots x_{1}^{s_{1}}$. Let $N\ge 2$. We focus on the terms  of the forms $x^{N-1}x \otimes x_{1}$ and $x_{1}x^{N-1}\otimes x_{N-1}$ in $\Delta(z)$. They only occur in $\Delta(\ell_{1,0,\ldots,0}x_{N})$ and $\Delta(\ell_{0,1,0,\ldots,1}x_{N-1}x_{1})$. We make the calculation:
\begin{align*}
\Delta(x_{N}) =& x_{N} \otimes 1 + ((N+1)x_{N-1}+ \binom{N+1}{2} x_{N-2}x_{1}+\ldots)x \otimes x_{1}+ \ldots  \\
& + \binom{N+1}{2} x_{1}^{}x^{N-1} \otimes x_{N-1}  + x^{N} \otimes x_{N}, \\
\Delta(x_{N-1}x_{1}) 
=& x_{N-1}x_{1} \otimes 1+ x_{N-1}x\otimes x_{1}+ \ldots + x_{1}x^{N-1} \otimes x_{N-1} +x^{N} \otimes x_{N-1}x_{1}.
\end{align*}
If $N=2$, then $z=\alpha z_{2}$ for some $0\neq \alpha\in \K$. If $N\ge 3$, then
\begin{align*}
\ell_{1,0,\ldots,0}(N+1)+\ell_{0,1,0,\ldots,1} &=0,\\
\ell_{1,0,\ldots,0}\binom{N+1}{2}+\ell_{0,1,0,\ldots,1} &=0.
\end{align*}
Since $\Char\K=0$, we have $\ell_{1,0,\ldots,0}=\ell_{0,1,0,\ldots,1}=0$. Thus $z=\sum_{(N-2)s_{N-2}+\ldots+s_{1}=N}^{} \ell_{0,0,s_{N-2},\ldots,s_{1}}\\ x_{N-2}^{s_{N-2}} \cdots x_{1}^{s_{1}}$. Now, if $N=3$, then $z=\ell_{0,0,3}x_{1}^{3}$ is not primitive because $\Char\K=0$. Thus $\ell_{0,0,3}=0$. If $N\ge 4$, then we focus on the terms of the form $\underline{\quad} \otimes x_{N-2}$ in $\Delta(z)$. Observe that they only occur in $\Delta(\ell_{0,0,1,\ldots,1,0}x_{N-2}x_{2})$ and $\Delta(\ell_{0,0,1,\ldots,0,2}x_{N-2}x_{1}^{2})$. Then
\begin{align*}
\ell_{0,0,1,\ldots,1,0}x_{2}x^{N-2}+\ell_{0,0,1,\ldots,0,2}x_{1}^{2}x^{N-2}=0.
\end{align*}
Thus $\ell_{0,0,1,\ldots,1,0}=\ell_{0,0,1,\ldots,0,2}=0$ and $z=\sum_{(N-3)s_{N-3}+\ldots+s_{1}=N}^{} \ell_{0,0,s_{N-3},\ldots,s_{1}} x_{N-3}^{s_{N-3}} \cdots x_{1}^{s_{1}}$. Similarly, we obtain inductively $\ell_{s_{N},\ldots,s_{1}}=0$ for all $s_{N},\ldots,s_{1}$ such that $Ns_{N}+\ldots+s_{1}=N$. Thus $z=0$ if $N\ge 3$. Therefore, $N=2$ and $z=\alpha z_{2}$ for $0\neq \alpha\in \K$.
\end{proof}

Let $\gr_{\cF^{\pm}} (\Tt_{\pm 1}^{cop})$ be the graded Hopf algebra associated to the filtration $\{\cF_n^{\pm}\}_{n=0}^{\infty}$.

\begin{theorem}\label{thm:Radfordbiproduct-grTtpm}
$\gr_{\cF^{\pm}}(\Tt_{\pm 1}^{cop})\cong B_{\pm}\# \K\Z$, where $B_{\pm}$ is a braided Hopf algebra in $ {}^{\K\Z}_{\K\Z}\mathcal{YD}$ determined by
\begin{align}
\label{relation1-B_pm}
[x_{1},y] &= \frac{1}{2} x_{1}^{2}, \qquad [z_{n},y] = -z_{n}x_{1}-\frac{n(n+1)!}{2^{n}}z_{2}x_{1}^{n-1}+z_{n+1}, \qquad n\ge 2,\\
\label{relation2-B_pm}
[x_{1},z_{n}] &=0, \qquad [z_{n},z_{m}]=0,  \qquad 2\le n< m.\\
x\cdot x_{1} &=x_{1}, \qquad x\cdot y=y+x_{1}, \qquad x\cdot z_{n}=z_{n}, \qquad n\ge 2\\
\rho(x_{1}) &=x\otimes x_{1}, \qquad \rho(y)=x\otimes y, \qquad \rho(z_{n})=x^{n}\otimes z_{n}, \qquad n\ge 2.
\end{align}
If $\Char\K=0$, then $B_{\pm}$ is the Nichols algebra of $V_{1,2}=\K\{y,z_2,x_{1}\}\in {}^{\K\Z}_{\K\Z}\mathcal{YD}$.
\end{theorem}

\begin{proof}
Note that $z_{n}\in \cF_{n-1}^{\pm}$ for $n\ge 2$. Then, by Lemma \ref{lem:redefinition-Ttpm}, the relations of $\operatorname{gr}_{\cF}(\Tt_{\pm 1}^{cop})$ are determined by (\ref{relation1-Ttpm}), (\ref{relation1-B_pm}) and (\ref{relation4-Ttpm}).

Thus $\operatorname{gr}_{\cF}(\Tt_{\pm 1}^{cop})\cong B_{\pm}\# \K\Z$ by Lemma \ref{lem:Radfordbiproduct}. 
Observe that $B_{\pm}$ is generated by $B(1)=V_{1,2}$. If $\Char\K=0$, then  $\Pp(B_{\pm})=B_{\pm}(1)=V_{1,2}$ by Lemma \ref{lem:first-term-coradicalfil}. Therefore, $B_{\pm}\cong\BN(V_{1,2})$.
\end{proof}

\begin{corollary}\label{cor:coradical-filtration-Ttpmcop}
Assume $\Char\K$ =0. Then $\{\cF_n^{\pm}\}_{n=0}^{\infty}$ is the coradical filtration of $\Tt_{\pm 1}^{cop}$ (and $\Tt_{\pm 1}$).
\end{corollary}

\begin{proof}
Note that the coradical filtration of $\Tt_{\pm 1}^{cop}$ is also the coradical filtration of $\Tt_{\pm 1}$. Thus the claim follows directly from Theorem \ref{thm:Radfordbiproduct-grTtpm}.
\end{proof}

\begin{remark}\label{rmk:B(V_1,2)}
(i) With the notations in Theorem \ref{thm:Radfordbiproduct-grTtpm},
using the formula of the braiding in ${}_{\K\Z}^{\K\Z}\mathcal{YD}$, we have the braided matrix of $V_{1,2}$ given by
\begin{align*}
c\left(
\begin{array}{lll}
x_{1} \otimes x_{1} & x_{1}\otimes y & x_{1}\otimes z_{2} \\
y \otimes x_{1}     & y \otimes y  & y\otimes z_{2}\\
z_{2} \otimes x_{1}  & z_{2}\otimes y & z_{2} \otimes z_{2}
\end{array}
\right)
=
\left(
\begin{array}{lll}
x_{1} \otimes x_{1} & (y+x_{1})\otimes x_{1}  & z_{2} \otimes x_{1}  \\
x_{1} \otimes y     & (y+x_{1})\otimes y      & z_{2} \otimes y\\
x_{1} \otimes z_{2}     & (y+2x_{1})\otimes z_{2} & z_{2}\otimes z_{2}
\end{array}
\right).
\end{align*}
Hence $\K\{x_1,y\}$ is of Jordan type (see e.g. \cite{CLW2009} and \cite{AAH12021}). The braiding of $V_{1,2}$ also appeared in \cite[Section 4.1.1]{AAH12021} with the ghost equal to $-4$.  If $\Char\K=0$, $\BN(V_{1,2})\cong B_{\pm}$ has infinite $\GK$-dimension and its defining relations are given by relations (\ref{relation1-B_pm}) and (\ref{relation2-B_pm}).\\

\noindent
(ii) Let $z_{2}':=z_{2}$ and
$z_{n+1}':=[y,z_{n}']_{c}=yz_{n}'-(y_{(-1)}\cdot z_{n}')y_{(0)}$ for $n\ge 2$. It is clear that $z_{n+1}'
=yz_{n}'-(x\cdot z_{n}')y=[y,z_{n}']$. Set $z_{0}=z_{1}=0$. We can prove that
\begin{align}
z_{n}'&=(-1)^{n}(z_{n}-(n+1)z_{n-1}x_{1}+\frac{n(n+1)}{2}z_{n-2}x_{1}^{2}), \qquad n\ge 2.
\end{align}
In case $\Char\K=0$, the relations of $B_{\pm}$ (or $\BN(V_{1,2})$) are equivalent to
\begin{align*}
x_{1}y &= yx_{1}+\frac{1}{2}x_{1}^{2}, \qquad yz_{n}' = z_{n}'y+z_{n+1}', \qquad n\ge 2,\\
x_{1}z_{n}' &= z_{n}'x_{1}, \qquad  z_{n}'z_{m}' = z_{m}'z_{n}', \qquad 2\le n\le m.
\end{align*}
Let $c_{n,k}\in \K$ for $1\le k\le n+1$, $c_{n,n}:=1$ and $c_{n,1}=c_{n,n+1}=0$. Then
\begin{align*}
\Delta(z_{n}') &=z_{n}' \otimes 1 + \sum_{k=2}^{n}c_{n,k} x_{1}^{n-k}\otimes z_{k}',
\end{align*}
where $c_{n,k}$ is determined by the recursion:
\begin{align*}
c_{n,k} &=c_{n-1,k-1}-\frac{n-1+k}{2}c_{n-1,k}, \qquad 2\le k\le n.
\end{align*}
More precisely,
\begin{align}
c_{n,k} &= (-\frac{1}{2})^{n-k}\binom{n-2}{k-2}(k+2)(k+3)\cdots (n+1), \qquad 2\le k\le n-1.
\end{align}
\end{remark}

\begin{remark}
In case $\Char\K=p>0$, $B_{\pm}$ is not a Nichols algebra of $V_{1,2}$.
\end{remark}

Let $\cF_n^{\operatorname{FdB}}:=F_{n}^{\operatorname{FdB}}\K[x^{\pm 1}]$ for $n \ge 0$, where
\begin{align*}
F_{n}^{\operatorname{FdB}}= \bigoplus_{(i_{m}-1)s_{m}+\ldots+(i_{2}-1)s_{2}+s_{1}\le n,\atop  i_{m}>\ldots>i_{2}\ge 2} \K  z_{i_{m}}^{s_{m}}\cdots z_{i_{2}}^{s_{2}}x_{1}^{s_{1}}
\end{align*}
For example,
\begin{align*}
\cF_{0}^{\operatorname{FdB}} &=\K [x^{\pm 1}],\\
\cF_{1}^{\operatorname{FdB}} &= (\K 1 \oplus \K z_{2} \oplus \K x_{1}) \K [x^{\pm 1}],\\
\cF_{2}^{\operatorname{FdB}} &= (\K 1 \oplus \oplus \K z_{2} \oplus \K x_{1} \oplus \K z_{3} \oplus \K z_{2}^{2} \oplus \K z_{2}x_{1 }\oplus \K x_{1}^{2}) \K [x^{\pm 1}].
\end{align*}
It is clear that $\cF_0^{\operatorname{FdB}} \subseteq \cF_1^{\operatorname{FdB}} \subseteq \ldots \subseteq \cF_n^{\operatorname{FdB}}\subseteq \ldots \subseteq \mathcal{H}_{\operatorname{FdB}}^{cop}$.
\begin{corollary}
Assume $\Char \K=0$. Then $\{\cF_n^{\operatorname{FdB}}\}_{n=0}^{\infty}$ is the coradical filtration of $\mathcal{H}_{\operatorname{FdB}}^{cop}$ (and $\mathcal{H}_{\operatorname{FdB}})$.
\end{corollary}
\begin{proof}
Note that $\cF_n^{\operatorname{FdB}}=\mathcal{H}_{\operatorname{FdB}}^{cop} \cap \cF_{n}^{\pm}$ for all $n\ge 0$. Thus the claim follows from \cite[Corollary 4.2.2 (a)]{R2012} and Corollary \ref{cor:coradical-filtration-Ttpmcop}.
\end{proof}

\subsection{The coradical filtration of $\Tt_{p}(p;1,\ldots,1)$}\qquad

We study the coradical filtration of $\Tt_{p}(p;1,\ldots,1)$. Let $\pi_{\pm}:\Tt_{\pm 1}^{cop}\rightarrow \Tt_{p}(p;1,\ldots,1)^{cop}$ be the projection. If there is no confusion, we still write $x$ for $\pi_{\pm}(x)$, $\forall x\in\Tt^{\pm}$.

 Let $\cF_n=\pi_{\pm}(\cF_n^{\pm})$. Then $\cF_n:=F_{n}\K[x]/(x^{p})$, where
\begin{align*}
F_{n}= \bigoplus_{r+(i_{m}-1)s_{m}+\ldots+(i_{2}-1)s_{2}+s_{1}\le n,\atop  i_{m}>\ldots>i_{2}\ge 2} \K y^{r}z_{i_{m}}^{s_{m}}\cdots z_{i_{2}}^{s_{2}}x_{1}^{s_{1}}
\end{align*}
\begin{lemma}
The filtration $\{\cF_n\}_{n=0}^{\infty}$ is a Hopf algebra filtration of $\Tt_{p}(p;1,\ldots,1)^{cop}$.
\end{lemma}

\begin{proof}
It is similar to the proof of Lemma \ref{lem:hopf-filtration}.
\end{proof}

\begin{lemma}\label{lem:1stcoradfiltration-f.d.}
Assume $\Char\K=p>3$. Then $\cF_1$ is the first term of the coalgebra filtration of $\Tt_{p}(p;1,\ldots,1)^{cop}$.
\end{lemma}
\begin{proof}
By the Taft-Wilson Theorem, it suffices to show that any non-trivial $(x^N,1)$-skew primitive element is contained in $F_1$, where $0\le N<p$. It is clear that $\Pp_{x,1}(\Tt_{p}(p;1,\ldots,1)^{cop})=\K\{1-x,y,x_1\}$.
Let $z\in \Pp_{x^{N},1}(\Tt_{p}(p;1,\ldots,1)^{cop})$ for $2\le N<p$. It remains to show that $z=0$ if $3\le N\le p-1$; and $z=\alpha z_{2}$ for some $\alpha\neq 0$ if $N=2$.
Note that $\{y^{r}x_{p-2}^{s_{p-2}}\cdots x_{1}^{s_{1}}x^{s}|~0\le r,s_{p-2},\ldots,s_{1},s\le p-1\}$ is a basis of $\Tt_{p}(p;1,1,\ldots,1)^{cop}$. The rest is similar to the proof of Lemma \ref{lem:first-term-coradicalfil}.
\end{proof}

Let $\gr_{\cF} (\Tt_{p}(p;1,\ldots,1)^{cop})$ be the graded Hopf algebra associated to the filtration $\{\cF_n\}_{n=0}^{\infty}$.

\begin{theorem}\label{thm:Ttpp1-biproduct}
Assume $\Char\K=p>3$. Then $\operatorname{gr}_{\cF}(\Tt_{p}(p;1,\ldots,1)^{cop})\cong B\# \K\Z_{p}$, where $B$ is of dimension $p^{p-1}$, determined by
\begin{align*}
[x_{1},y] &= \frac{1}{2} x_{1}^{2}, \qquad [z_{n},y] = -z_{n}x_{1}-\frac{n(n+1)!}{2^{n}}z_{2}x_{1}^{n-1}+z_{n+1}, \quad 2\le n\le p-2, \qquad z_{p-1} =0,\\ 
[x,x_{1}] &=0, \qquad [x,z_{n}] =0, \qquad[x_{1},z_{n}]=0, \qquad [z_{n},z_{m}]=0,  \qquad 2\le n\le  m\le p-2,\\
x_{1}^{p}&=0, \qquad y^{p}=0, \qquad z_{n}^{p}=0, \qquad 2\le n\le p-2,\\
x\cdot x_{1} &=x_{1}, \qquad x\cdot y=y+x_{1}, \qquad x\cdot z_{n}=z_{n}, \qquad 2\le n\le p-2,\\
\rho(x_{1}) &=x\otimes x_{1}, \qquad \rho(y)=x\otimes y, \qquad \rho(z_{n})=x^{n}\otimes z_{n}, \qquad 2\le n\le p-2.
\end{align*}
Furthermore, $B\cong \B(V_{1,2})$, where $V_{1,2}=\K\{y,z_2,x_1\}\in {}^{\K\Z_{p}}_{\K\Z_{p}}\mathcal{YD}$.

\end{theorem}
\begin{proof}
Similar to the proof of Theorem \ref{thm:Radfordbiproduct-grTtpm}, we have $\gr_{\cF}(\Tt_{p}(p;1,\ldots,1)^{cop})\cong B\# \K\Z_{p}$. It is clear that $B$ is generated by $B(1)=V_{1,2}$. By Lemma \ref{lem:1stcoradfiltration-f.d.}, $\Pp(B)=B(1)=V_{1,2}$. Therefore, $B\cong \B({V_{1,2}})$.
\end{proof}
\begin{remark}
With the notations in Remark \ref{rmk:B(V_1,2)} (ii), if $\Char\K=p>3$, then the relations of $\BN(V_{1,2})$ are equivalent to
\begin{align*}
x_{1}y &= yx_{1}+\frac{1}{2}x_{1}^{2}, \qquad yz_{n}' = z_{n}'y+z_{n+1}', \qquad 2\le n\le p-2,\\
z_{p-1}' &=0, \qquad x_{1}z_{n}' = z_{n}'x_{1}, \qquad  z_{n}'z_{m}' = z_{m}'z_{n}', \qquad 2\le n\le m\le p-2,\\
x_{1}^{p} &=0, \qquad y^{p}=0, \qquad z_{n}'^{p}=0, \qquad 2\le n\le p-2.
\end{align*}

We mention that this braided Hopf algebra has already appeared in \cite[Proposition 4.10]{AAH22021}.
\end{remark}

Note that when $\Char\K=0$, $\BN(V_{1,2})$ is of infinite $\GK$-dimension, while $\BN(V_{1,2})$ is finite-dimensional if $\Char\K=p>3$.
\begin{remark}
If $\Char\K=3$, then  $\Tt_{3}(3;1)\cong \operatorname{gr}(\Tt_{3}(3;1))\cong \B(V)\# \K\Z_{3}$, where $V=\K y\oplus \K x_{1}$ and $\B(V)$ is determined by
\begin{align*}
x_{1}y &= yx_{1}+\frac{1}{2}x_{1}^{2}, \qquad x_{1}^{3}=0, \qquad y^{3}=0.
\end{align*}
We mention that $\Tt_{3}(3;1)$ has appeared in \cite[Corollary 3.14]{CLW2009}, and its Nichols algebra is of Jordan type (see also \cite[Theorem 3.5]{CLW2009} and \cite[Section 3.5]{AS2002}).
\end{remark}

\begin{corollary}
Assume $\Char\K=p>3$. Then $\{\cF_{n}\}_{n=0}^{\infty}$ is the coalgebra filtration of $\Tt_{p}(p;1,\ldots,1)^{cop}$ (and $\Tt_{p}(p;1,\ldots,1)$).
\end{corollary}
\begin{proof}
It follows directly from Theorem \ref{thm:Ttpp1-biproduct}.
\end{proof}

\section*{\textbf{ACKNOWLEGEMENT}}
The first author thanks both the China Scholarship Council (No. 201906140164) and BOF-UHasselt  for its financial support during his research study at the University of Hasselt.

\bibliographystyle{plain}

\end{document}